\documentclass[english,11pt]{amsart}

\usepackage[english]{babel}
\usepackage[T1]{fontenc}
\usepackage{lmodern}

\usepackage{MnSymbol}
\usepackage[centering]{geometry}

\theoremstyle{plain}
\newtheorem{theo}{Theorem}
\newtheorem{lemm}[theo]{Lemma}
\newtheorem{prop}[theo]{Proposition}
\newtheorem{coro}[theo]{Corollary}

\theoremstyle{definition}
\newtheorem{defi}[theo]{Definition}

\theoremstyle{remark}

\usepackage{tikz}
\usetikzlibrary{arrows.meta,positioning,calc,patterns,shapes.geometric}
\usepackage[colorlinks=true,linkcolor=blue!50!black,citecolor=blue!50!black,urlcolor=blue!50!black]{hyperref}
\usepackage{microtype}
\usepackage{enumerate}
\usepackage{graphicx}
\usepackage{color}
\usepackage{bm}
\usepackage{mathtools}
\usepackage[dvipsnames]{xcolor}
\definecolor{FlatRed}{RGB}{231,76,60}
\definecolor{FlatGreen}{RGB}{46,204,113}
\definecolor{FlatBlue}{RGB}{52,152,219}
\definecolor{FlatYellow}{RGB}{241,196,15}
\colorlet{FlatViolet}{FlatRed!50!FlatBlue}
\colorlet{FlatBrown}{FlatRed!50!FlatGreen}
\colorlet{FlatOrange}{FlatRed!50!FlatYellow}
\colorlet{FlatCyan}{FlatGreen!50!FlatBlue}

\usepackage{enumitem}
\usepackage{textcomp}
\usepackage[normalem]{ulem}

\usepackage{hyperref}
\hypersetup{
    colorlinks,
    linkcolor={FlatRed},
    citecolor={FlatGreen},
    urlcolor={FlatBlue}
}

\title{Exponential decay of fractional moments in the \(H^{2|2}\) model on \(\mathbb Z^2\)}

\author{Jinglin Wang}
\address{Institut de Recherche Mathématique Avancée, Université de Strasbourg, Strasbourg, France}
\email{jinglin.wang.math@outlook.com}

\author{Xiaolin Zeng}
\address{Institut de Recherche Mathématique Avancée, Université de Strasbourg, Strasbourg, France}
\email{zeng@math.unistra.fr}

\begin{document}

\begin{abstract}
For the \(H^{2|2}\) model on \(\mathbb Z^2\), we show that, at energy zero, the fractional moments of Green functions, \(\mathbb E[ G(x,y)^s]\) decay exponentially in the distance between \(x\) and \(y\), for every constant edge weight \(W>0\) and every \(0<s<1/2\). Here \( G\) is the infinite-volume Green function of the associated random Schr\"odinger operator. The main ingredient is a quantitative lower bound on the derivative of the mean effective conductance with respect to the edge weights.
\end{abstract}


\maketitle
% \tableofcontents

\section{Introduction}

The supersymmetric hyperbolic sigma model \(H^{2|2}\), introduced in \cite{Zirnbauer1991}, provides a model for studying localization and delocalization in disordered systems. The existence of a quasi-diffusive phase at large coupling in dimensions at least three was shown in \cite{DisertoriSpencerZirnbauer2010}, while localization at sufficiently small coupling was established in \cite{DisertoriSpencer2010}. Its effective field was identified as the limiting local times profile of the Vertex-Reinforced Jump Process in \cite{SabotTarres2015}, leading to a random walk in random environment point of view for the VRJP. Its representation through a random Schrödinger operator was developed on finite graphs in \cite{SabotTarresZeng2017}, and on infinite graphs in \cite{Sabot2019}. Subsequent developments include monotonicity in the edge weights \cite{Poudevigne2024,HuangZeng2026} and transience for models with long-range interactions \cite{DisertoriMerklRolles2025}.

We consider the \(H^{2|2}\)-model on \(\mathbb Z^2\), with constant edge weight \(W>0\). The question is whether the fractional moments \(\mathbb E[ G(x,y)^s]\), \(0<s<1/2\), decay exponentially in \(|x-y|_1\) for large \(W>0\). Here \( G\) is the zero-energy Green function of the associated random Schrödinger operator \(H_\beta\), introduced in \cite{SabotTarresZeng2017}, defined as increasing finite-volume limits with Dirichlet boundary conditions. We also consider finite regions with wired boundary conditions, i.e. whose exterior is identified to a single vertex \(\delta\), with parallel edge weights added and the effective field pinned by \(u_\delta=0\).

In two dimensions, polynomial decay of \(\mathbb E[e^{s u_y}]\), \(0<s<\frac{1}{2}\), for the field pinned at the origin in wired boxes was shown in \cite{Sabot2021,KozmaPeled2021} under uniformly bounded positive edge weights. Together with the infinite-volume representation in \cite{Sabot2019}, these estimates imply recurrence of the VRJP with constant weights. Exponential decay was available at sufficiently small \(W\), corresponding to strong reinforcement, from the localization estimates in \cite{DisertoriSpencer2010} and the fractional-moment approach in \cite{CollevecchioZeng2021}.

We show exponential decay of the zero-energy Green-function fractional \(s \) moments on \(\mathbb Z^2\) for every constant \(W>0\) and every \(0<s<1/2\), without a smallness assumption on \(W\). The estimates hold uniformly for finite-volume Dirichlet Green functions and pass to the infinite-volume limit. For the full Green matrix on a wired finite region, there is an additional term that decays exponentially in the sum of the distances of the two vertices to the boundary. Section~2 gives the statements and an overview of the proof. Section 3 recall results of the \(H^{2|2}\)-model on finite graphs. Section 4 derives the differential inequality for the mean effective conductance. In Section 5, we combines this inequality with a Mermin--Wagner estimate to obtain polynomial decay across an annular
block in \eqref{eq:Z-r-polynomial-bound}, Section 6 use the Simon--Lieb argument to obtain exponential decay.

\section{Definitions, main results, and proof overview}
We recall  the random-potential and effective-field representations of the model on \(\mathbb{Z}^2 \) as in \cite{Sabot2019}. Details on the extension to infinite graphs are explained in Subsection~\ref{subsec:infinite-graphs-wired-coupling}.

On the square lattice \(\mathbb{Z}^2= (V,E) \), if \(x,y\in V \) are connected by an edge in \(E \), we denote \(x\sim y \). To each edge \(x\sim y \) we associate an edge weight \(W_{xy} >0 \), and we denote by \(W \) the collection of these edge weights. By \cite[Proposition 1]{Sabot2019}, there exists a real random field \((\beta_{x})_{x\in V} \) whose law is denote \(\nu_{V}^{W} \),  such that for any finite subset \(U\subset V \), we have
\begin{equation}\label{eq:Laplace-transform-beta-on-U}
\begin{aligned}
&\int e^{-\sum_{i\in U} \lambda_i \beta_i} \nu_V^W(d\beta)\\
&=e^{-\sum_{\substack{i\sim j, \\ i,j\in U}} W_{i,j}(\sqrt{(1+\lambda_i)(1+\lambda_j)}-1)
-\sum_{\substack{i\sim j,\\ i \in U, j\notin U}} W_{i,j}(\sqrt{1+\lambda_i}-1)}{1\over \prod_{i\in U} \sqrt{1+\lambda_i}}.
\end{aligned}\end{equation}
The random potential \((\beta_{x})_{x\in V} \) defines a random Schrödinger operator \(H_{\beta} \) on its domain \(\mathcal{D}(H_{\beta}) =\left\{f\in\ell^2(\mathbb Z^2): \sum_x\beta_x^2|f(x)|^2<\infty\right\} \), by
\[(H_{\beta}f)(x) = 2\beta_{x}f(x)-\sum_{y:y\sim x} W_{xy} f(y),\ \forall f\in \mathcal{D}(H_{\beta}).\]
Moreover, \cite[Theorem 1]{Sabot2019} defines two notion of inverse of \(H_{\beta} \), namely, \(\widehat{G}(x,y) \) and \(G(x,y) = \widehat{G}(x,y)+ \frac{1}{2\gamma} \psi(x)\psi(y) \), defined in \eqref{eq:infinite-dirichlet-green-definition} and \eqref{eq:infinite-wired-green-definition}, where \(\gamma \) is a Gamma random variable with density \(e^{-t}/\sqrt{\pi t}\) on \(t>0\), independent of \(\beta \) and \((\psi(x))_{x\in V} \) is another real random field, can be think of as a generalized eigenfunction of \(H_{\beta} \) at energy 0, obtained as martingale limit of \(H^{2|2}\)-model on a sequence of finite exhausting subsets of \(\mathbb{Z}^2 \).

Consider a sequence of increasing finite connected subsets \(V_n\uparrow V \). To each \(V_n \), let \(\partial^{-} V_{n} = \{x\in V_{n}, \exists y\notin V_n, x\sim y\} \) be its inner boundary, let \(\delta_{n} \notin V \) be an auxiliary vertex, we associate to \(V_{n} \) a graph \((\widetilde{V}_{n},\widetilde{E}_{n}) \), where \(\widetilde{V}_{n}= V_n\cup\{\delta_{n} \} \) and
\[\widetilde{E}_{n}= \{\{x,y\}\in E,\ x,y\in V_{n} \} \cup \{\{x,\delta_{n}\},\ x\in \partial^{-} V_{n}\}.\] Also define
\[W_{x\delta_{n}}=\sum_{y:y\sim x,y\notin V_{n}}W_{xy}.\]
By \cite[Section 4]{Sabot2019}, on each \(\widetilde{V}_{n} \), we can define random potential \((\beta^{(n)}_{x})_{x\in \widetilde{V}_{n}} \), in such a way that they are coupled with \((\beta_{x})_{x\in V} \). In particular, \(\forall x\in V, \forall n\ge 1, \ \beta^{(n)}_{x}=\beta_{x} \) once \(x\in V_{n} \); and for each \(n\ge 1 \), \(\beta_{\delta_{n}}:=\beta^{(n)}_{\delta_{n}} \) is given by \eqref{eq:wired-potential-coupling} as a measurable function of \( (\beta_{x})_{x\in V_{n}} \) and \(\gamma \).

This allow us to define \(H^{(n)}_{\beta} \), a random Schrödinger matrix on \(\ell^2(\widetilde{V}_{n}) \), by
\[(H_{\beta}^{(n)}f)(x) = 2 \beta_{x}f(x) -  \sum_{y\in \widetilde{V}_{n}, y\sim x}W_{xy}f(y) \]
and we denote \(G_{\beta}^{(n)}= (H_{\beta}^{(n)})^{-1} \), which exist a.s. Define the effective \(H^{2|2} \) field, \((u^{(n)}_{x})_{x\in V_{n}} \) with pinning at \(\delta_{n} \) by
\[e^{u^{(n)}_{x}}= \frac{G^{(n)}_{\beta}(\delta_{n},x)}{G^{(n)}_{\beta}(\delta_{n},\delta_{n})},\ x\in V_{n}\]
where \(G_{\beta}^{(n)}(\delta_{n},\delta_{n}) = \frac{1}{2\gamma} \) for all \(n \ge 1\). The joint law of \((u^{(n)}_{x})_{x\in V_{n}} \) is denoted by \(\mu_{\delta_{n},\widetilde{V}_{n}}^{W}(du) \), called the effective \(H^{2|2}\)-model on \(\widetilde{V}_{n} \) pinned at \(\delta_{n} \). It describes the mixing measure of the Vertex Reinforced Jump Process on \(\widetilde{V}_{n} \) with initial position \(\delta_{n} \), see \cite{SabotTarres2015,SabotTarresZeng2017}.

Set
\begin{equation}\label{eq:dirichlet-green-definition}
\widehat{G}^{(n)}_\beta
:=
\left((H_\beta^{(n)})_{V_n\times V_n}\right)^{-1}
=
\left((H_\beta)_{V_n\times V_n}\right)^{-1}.
\end{equation}
By \cite[Theorem~1(i)]{Sabot2019}, for every \(x,y\in V\), the entries \(\widehat{G}^{(n)}_\beta(x,y)\), defined for all sufficiently large \(n\), increase almost surely to a finite limit
\begin{equation}\label{eq:infinite-dirichlet-green-definition}
\widehat G(x,y):=\lim_{n\to\infty}\widehat{G}^{(n)}_\beta(x,y),
\end{equation}
which does not depend on the exhaustion. Also
\begin{equation}\label{eq:wired-dirichlet-green-decomposition}
G_\beta^{(n)}(x,y)
=
\widehat{G}^{(n)}_\beta(x,y)
+\frac{1}{2\gamma}e^{u_x^{(n)}+u_y^{(n)}},
\qquad x,y\in V_n.
\end{equation}
It is shown in \cite{Sabot2019} that \(e^{u^{(n)}_{x}} \to \psi(x) \) a.s. and thus we define
\begin{equation}\label{eq:infinite-wired-green-definition}
G(x,y):=\lim_{n\to\infty}G_\beta^{(n)}(x,y)
=
\widehat G(x,y)+\frac{1}{2\gamma}\psi(x)\psi(y).
\end{equation}
On \(\mathbb{Z}^2 \)  with constant edge weight, by e.g. \cite{Sabot2021} and  \cite[Theorem~1(iv)]{Sabot2019}, \(\psi\equiv0\) almost surely, thus \(G(x,y)=\widehat{G}(x,y) \).

\begin{theo}
For all \(W > 0 \), consider the \(H^{2|2} \) model on \(\mathbb{Z}^2 \) with constant edge weight \(W \), for any \(0 < s<\frac{1}{2} \), there exist constants \(C_{W,s},c_{W,s} > 0 \), such that, for all \(x,y\in V \),
\begin{equation}\label{eq:main-green-exponential-decay}
\mathbb{E}\left( G(x,y)^{s} \right) \le  C_{W,s}e^{-c_{W,s}|x-y|_{1}},
\end{equation}
where \(|x-y|_{1} \) is the graph distance on \(\mathbb{Z}^2 \).
\end{theo}

We also have a finite graph version:
\begin{theo}
Fix \(W>0\) and \(0<s<1/2\). there exist constants \(C_{W,s},c_{W,s}>0\) such that for every nonempty finite set \(V_n\subset \mathbb{Z}^2 \) with wired boundary condition, and all \(x,y\in V_{n}\),
\[\mathbb E[\widehat{G}^{(n)}_{\beta}(x,y)^s] \le C_{W,s}e^{-c_{W,s}|x-y|_1},\]
and with \(d_{V_{n}}(x)=d(x,V_{n}^{c}) =\min_{z\notin V_{n}}|x-z|_1\),
\[\mathbb E[G^{(n)}_\beta(x,y)^s] \le C_{W,s}\left( e^{-c_{W,s}|x-y|_1} +e^{-c_{W,s}(d_{V_n}(x)+d_{V_n}(y))} \right).\]
\end{theo}

\subsection{Proof overview}
For an integer \(r\ge0\), write
\[B(x,r)=\{z\in\mathbb Z^2:|z-x|_1\le r\}, \qquad S_r(x)=\{z\in\mathbb Z^2:|z-x|_1=r\}.\]
Fix an integer \(r_0\ge32\) and let \(R=2^nr_0\), where \(n\ge1\) is an integer.

Consider a ball \(B(x,R) \) centered at \(x\in V \), with its exterior wired to \(\delta\). The condition \(r_0\ge 32 \) ensures that \eqref{eq:Z-r-polynomial-bound} applies at every dyadic scale.

We cut \(B(x,R) \) into a ball \(B(x,r_0) \) and annular blocks \(\mathcal{A}_{r_0},\mathcal{A}_{2r_0},\cdots, \mathcal{A}_{2^{n-1}r_0} \) where each \(\mathcal{A}_{2^{\ell}r_0} \) is the induced sub-graph of \(\mathbb{Z}^2 \) with vertex set
\[V(\mathcal{A}_{2^{\ell}r_0}) = \{v\in V,\ 2^{\ell}r_0+2\le |v-x|_{1}\le 2^{\ell+1}r_0\}\]
and we contract all the vertices on \(S_{2^{\ell}r_0+1} = \{v\in V:\ |v-x|_{1}=2^{\ell}r_0+1\} \) into a vertex called \(a_{\ell} \). \(B(x,R) \), after such contraction, becomes a chain of annular blocks (see Fig.~\ref{fig:contracted-ball}), start with \(B(x,r_0) \), whose boundary is all connected to \(a_0 \), then \(a_0 \) is connected to \(\mathcal{A}_{r_0} \) through the inner lozenge \(S_{r_0+2}(x) \) of \(\mathcal{A}_{r_0} \), and the outer lozenge \(S_{2 r_0}(x) \) of \(\mathcal{A}_{r_0} \) is connected to \(a_1 \), etc, until \(\mathcal{A}_{2^{n-1}r_0} \). Denote by \(\check{B}(x,R) \) this contracted graph. The figures \ref{fig:adjacent-annular-blocks} below demonstrates how blocks are connected.

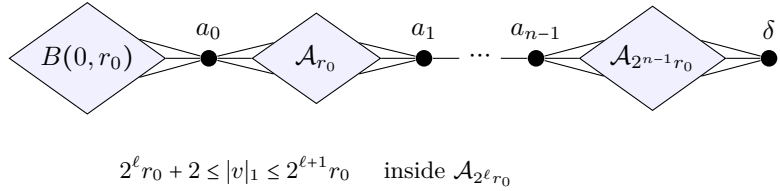
\begin{figure}[htbp]
  \centering
\begin{tikzpicture}[
    block/.style={
        draw, diamond, aspect=1.4,
        inner sep=2pt,
        fill=blue!6,
        minimum width=1.7cm,
        minimum height=1.2cm,
        align=center
    },
    junction/.style={
        circle, fill=black, inner sep=2.2pt
    },
    every node/.style={font=\small}
]
\node[block] (core) {$B(0,r_0)$};
\node[junction, right=0.45cm of core,
      label=above:$a_0$] (a0) {};
\node[block, right=0.45cm of a0] (A0)
      {$\mathcal A_{r_0}$};
\node[junction, right=0.45cm of A0,
      label=above:$a_1$] (a1) {};
\node[right=0.35cm of a1] (dots) {$\cdots$};
\node[junction, right=0.35cm of dots,
      label=above:$a_{n-1}$] (an) {};
\node[block, right=0.45cm of an] (An)
      {$\mathcal A_{2^{n-1}r_0}$};
\node[junction, right=0.45cm of An,
      label=above:$\delta$] (delta) {};

% Several lines indicate multiple boundary connections.
\foreach \angle in {-20,0,20} {
    \pgfmathtruncatemacro{\westangle}{180-\angle}
    \draw (core.\angle) -- (a0);
    \draw (a0) -- (A0.\westangle);
    \draw (A0.\angle) -- (a1);
    \draw (an) -- (An.\westangle);
    \draw (An.\angle) -- (delta);
}
\draw (a1) -- (dots);
\draw (dots) -- (an);

\node[below=0.55cm of A0, align=center, font=\footnotesize]
    {$2^{\ell}r_0+2\le |v|_1\le 2^{\ell+1}r_0$
     \quad inside $\mathcal A_{2^\ell r_0}$};
\end{tikzpicture}
\caption{After translating the centre to \(0\), contraction
of the layers \(S_{2^\ell r_0+1}\) turns the wired ball into
a chain of annular blocks.}
  \label{fig:contracted-ball}
\end{figure}

\begin{figure}[tbp]
  \centering
% Perspective schematic; rings depict only the surviving lattice layers.
\begin{tikzpicture}[x=1cm,y=1cm,font=\small,
  innerfan/.style={blue!70!black,line width=.45pt},
  outerfan/.style={orange!85!black,line width=.45pt}]
  \coordinate (left) at (0,0);
  \coordinate (joint) at (6,0);
  \coordinate (right) at (12,0);
  % Two annular sheets, viewed obliquely. Faint lozenges and segments suggest the lattice
  % region, not additional long-range physical edges.
  \foreach \c/\L/\R/\sz in {3/left/joint/1,9/joint/right/1.35}{
    % Rear portions of representative physical-edge fans.
    \foreach \ang in {110,150,190,230,270}{
      \draw[innerfan,dashed,opacity=.45] (\L) -- ({\c+\sz*.42*cos(\ang)/(abs(cos(\ang))+abs(sin(\ang)))},{\sz*.78*sin(\ang)/(abs(cos(\ang))+abs(sin(\ang)))});
      \draw[outerfan,dashed,opacity=.45] ({\c+\sz*.88*cos(\ang)/(abs(cos(\ang))+abs(sin(\ang)))},{\sz*1.65*sin(\ang)/(abs(cos(\ang))+abs(sin(\ang)))}) -- (\R);
    }
    \path[fill=cyan!12,even odd rule]
      ({\c+\sz*.88},0) -- (\c,{\sz*1.65}) --
      ({\c-\sz*.88},0) -- (\c,{-\sz*1.65}) -- cycle
      ({\c+\sz*.42},0) -- (\c,{\sz*.78}) --
      ({\c-\sz*.42},0) -- (\c,{-\sz*.78}) -- cycle;
    \foreach \rx/\ry in {.57/1.07,.72/1.36}{
      \draw[cyan!40!black,opacity=.3] ({\c+\sz*\rx},0) -- (\c,{\sz*\ry}) --
      ({\c-\sz*\rx},0) -- (\c,{-\sz*\ry}) -- cycle;
    }
    \foreach \ang in {0,30,...,330}{
      \draw[cyan!40!black,opacity=.3]
        ({\c+\sz*.42*cos(\ang)/(abs(cos(\ang))+abs(sin(\ang)))},{\sz*.78*sin(\ang)/(abs(cos(\ang))+abs(sin(\ang)))}) --
        ({\c+\sz*.88*cos(\ang)/(abs(cos(\ang))+abs(sin(\ang)))},{\sz*1.65*sin(\ang)/(abs(cos(\ang))+abs(sin(\ang)))});
    }
    \draw[blue!70!black,thick] ({\c+\sz*.42},0) -- (\c,{\sz*.78}) --
      ({\c-\sz*.42},0) -- (\c,{-\sz*.78}) -- cycle;
    \draw[orange!85!black,thick] ({\c+\sz*.88},0) -- (\c,{\sz*1.65}) --
      ({\c-\sz*.88},0) -- (\c,{-\sz*1.65}) -- cycle;
    % Front portions of the fans, drawn above the annular sheet.
    \foreach \ang in {-70,-30,10,50,90}{
      \draw[innerfan] (\L) -- ({\c+\sz*.42*cos(\ang)/(abs(cos(\ang))+abs(sin(\ang)))},{\sz*.78*sin(\ang)/(abs(cos(\ang))+abs(sin(\ang)))});
      \draw[outerfan] ({\c+\sz*.88*cos(\ang)/(abs(cos(\ang))+abs(sin(\ang)))},{\sz*1.65*sin(\ang)/(abs(cos(\ang))+abs(sin(\ang)))}) -- (\R);
    }
    \foreach \ang in {-70,-30,10,50,90,110,150,190,230,270}{
      \fill[blue!70!black] ({\c+\sz*.42*cos(\ang)/(abs(cos(\ang))+abs(sin(\ang)))},{\sz*.78*sin(\ang)/(abs(cos(\ang))+abs(sin(\ang)))}) circle(1.1pt);
      \fill[orange!85!black] ({\c+\sz*.88*cos(\ang)/(abs(cos(\ang))+abs(sin(\ang)))},{\sz*1.65*sin(\ang)/(abs(cos(\ang))+abs(sin(\ang)))}) circle(1.1pt);
    }
  }
  \node at (3,2.25) {$\mathcal A_r$};
  \node at (9,2.85) {$\mathcal A_{2r}$};
  \foreach \p in {left,joint,right}{\fill (\p) circle(2.7pt);}
  \node[below=9pt,fill=white,inner sep=2pt] at (left) {$a_\ell$};
  \node[below=9pt,fill=white,inner sep=2pt] at (joint) {$a_{\ell+1}$};
  \node[below=9pt,fill=white,inner sep=2pt] at (right) {$a_{\ell+2}$};
  \node[below=28pt,align=center,font=\scriptsize] at (left) {image of\\$S_{r+1}$};
  \node[below=28pt,align=center,font=\scriptsize] at (joint) {image of $S_{2r+1}$};
  \node[below=28pt,align=center,font=\scriptsize] at (right) {image of\\$S_{4r+1}$};
  % Leaders identify surviving layers, not the contracted layers.
  \foreach \c/\inn/\out/\sz in {3/{r+2}/{2r}/1,9/{2r+2}/{4r}/1.35}{
    \draw[blue!70!black] ({\c-\sz*.126},{-\sz*.546}) -- ({\c-\sz*.95},{-\sz*1.95});
    \node[below,text=blue!70!black] at ({\c-\sz*.95},{-\sz*1.95}) {$S_{\inn}$};
    \draw[orange!85!black] ({\c+\sz*.264},{-\sz*1.155}) -- ({\c+\sz*.95},{-\sz*1.95});
    \node[below,text=orange!85!black] at ({\c+\sz*.95},{-\sz*1.95}) {$S_{\out}$};
  }
  \node[align=center,font=\scriptsize] at (6,-3.55)
    {};
\end{tikzpicture}
\caption{Let $r=2^\ell r_0$. Two adjacent annular blocks and their boundary vertices $a_\ell$, $a_{\ell+1}$, and $a_{\ell+2}$.}
\label{fig:adjacent-annular-blocks}
\end{figure}
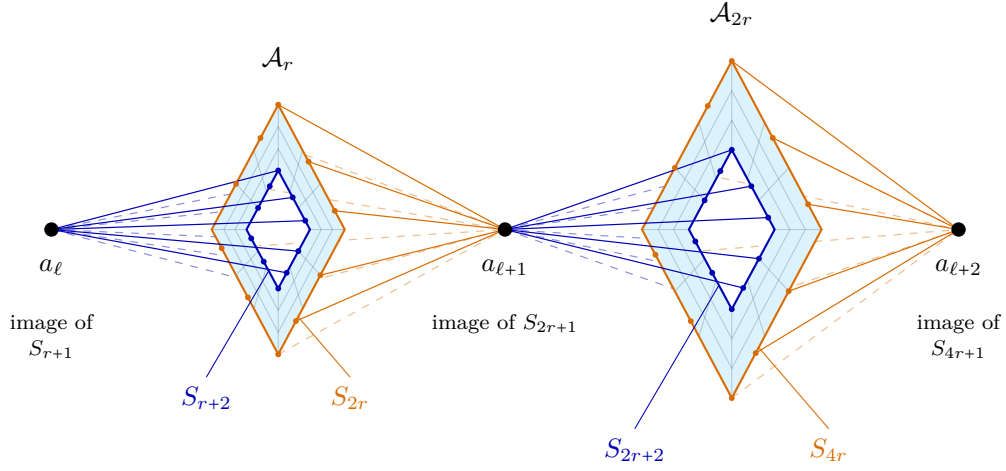

\begin{defi}\label{def:wired-annular-block}
For \(0\le\ell<n\), set \(r=2^\ell r_0\), with \(a_n=\delta\). Define the wired annular block \(\widetilde{\mathcal A}_r\) as the subgraph of \(\check B(x,R)\) induced by
\[
V(\widetilde{\mathcal A}_r)
=
\{z\in\mathbb Z^2:r+2\le |z-x|_1\le2r\}
\cup\{a_\ell,a_{\ell+1}\}.
\]
Its edge weights are inherited from the contraction: each weight is the sum of the weights of the original lattice edges represented by that edge, i.e.
\[W_{a_{\ell}z}=\sum_{y:y\sim z,y\in S_{2^{\ell}r_0+1} (x)} W_{zy}.\]
\end{defi}

The \(H^{2|2}\)-model on \(\check B(x,R)\) decomposes into independent relative fields on these blocks, as shown in Lemma~\ref{lem:block-field-independence}. The main estimate is
\[\mathbb{E}_{\mu_{a_{\ell+1},V(\widetilde{\mathcal{A}}_{2^{\ell}r_0})}^{W}}(e^{\frac{1}{2}u_{a_{\ell}}}) \le C_{W} (2^{\ell} r_0)^{-\alpha(W)}.\]
Here the constants \(C_{W} \) and \(\alpha(W) \) are given in \eqref{eq:Z-r-decay-constants}, and each \(a_{\ell},a_{\ell+1} \) represent \(O(2^{\ell}) \) vertices of \(V(\mathbb{Z}^{2}) \). Multiplying the bounds on the independent blocks, then use monotonicity of \(H^{2|2} \)-model, gives the following bound on the original wired ball, with possibly different constants \(C_W,c_W>0\),
\[\mathbb{E}_{\mu_{\delta,B(x,R)\cup\{\delta\}}^{W}}(e^{\frac{1}{2}u_{x}}) \le C_{W} \exp(-c_{W}(\log R)^2).\]
Adapting the finite-size criterion of Simon and Lieb \cite{Simon1980,Lieb1980}, see also \cite[Section~4.a]{ASFH2001}, to the \(H^{2|2}\)-model, we can conclude exponential decay.

To bound \(\mathbb{E}_{\mu_{a_{\ell+1},V(\widetilde{\mathcal{A}}_{2^{\ell}r_0})}^{W}}(e^{\frac{1}{2}u_{a_{\ell}}}) \), fix \(r=2^\ell r_0\), take \(\widetilde{\mathcal A}_r\) from Definition~\ref{def:wired-annular-block}, and write \(V=V(\widetilde{\mathcal A}_r)\), \(E=E(\widetilde{\mathcal A}_r)\), \(a=a_\ell\), and \(\delta=a_{\ell+1}\). Give each edge \(e\in E\) weight \(wa_e\), where \(a_e\) is the number of original lattice edges merged into \(e\), so that \(1\le a_e\le3\). The edge coefficients \(a_e=a_{ij}\), for \(e=\{i,j\}\), should not be confused with the vertex \(a\).

For the effective field pinned at \(\delta\), set
\[
Z_r(w):=\mathbb E_{\mu_{\delta,V}^{(wa_e)_{e\in E}}}
\left[e^{\frac{u_a}{2}}\right].
\]
We estimate \(Z_r(w)\) through a mean effective conductance\footnote{We are grateful to Poudevigne for suggesting that we work with the mean effective conductance, using the field on the graph obtained by identifying $a$ and $\delta$. His comments largely inspired the approach developed in Sections~4 and~5, simplifying the argument and its presentation compared with the first arXiv version.}. Identify the two boundary vertices \(a,\delta\) to a single vertex \(\rho\), let \(\mathring V =(V\setminus \{a,\delta\})\cup\{\rho\}\) and \(\mathring W \) be the edge weight obtained by identifying \(a,\delta \). Let \(\mathring u\) be the field on the graph \(\mathring V \) pinned at \(\rho\). Use this field to define the conductances \(W_{ij}e^{\mathring u_{i}+\mathring u_{j}} \) in \eqref{eq:C-from-W-and-U} on the original annulus, where \(a\) and \(\delta\) remain distinct. Let \(K\) be the effective conductance between \(a,\delta \), and write
\[M_r(w)=\mathbb E_{\mu_{\rho,\mathring V}^{\mathring W}}[K].\]
Conditioning on the potential on \(V\setminus\{a,\delta\}\) and taking the Schur complement at \(a,\delta\) bounds \(\mathbb E[e^{u_a/2}\mid\beta_{V\setminus\{a,\delta\}}]\) in terms of \(K^{1/4}\). Applying Jensen's inequality gives, by Lemma~\ref{lem:Z-bound-by-M},
\[Z_r(w)\le \frac{\Gamma(1/4)}{2^{1/4}\sqrt\pi}\,M_r(w)^{1/4}.\]

To bound \(M_r(w)\), choose a deterministic cutoff \(f\) equal to one near the inner boundary and zero near the outer boundary. Derivative w.r.t \(w \) gives Lemma~\ref{lem:M-differential-inequality}:
\[M_r(w)^2 \le (w+1)M_r'(w)D_f(w),\]
where
\[D_f(w)= \mathbb E_{\beta_{\widehat V}} \left[\sum_{\{i,j\}\in E} W_{ij}\left(\frac{C_{ij}}{W_{ij}}\right)^{1/4} (f(i)-f(j))^2\right].\]
Here \(\mathbb E_{\beta_{\widehat V}}\) denotes expectation under the law of \(\beta \) field on \(\widehat{V}=V\setminus \{a,\delta\} \), which also determines \(\mathring u\).

For every edge \(\{i,j\}\in E\) with \(f(i)\ne f(j)\), both endpoints belong to the bulk  \(\mathcal B_r\) of the annular block, defined in \eqref{eq:annular-bulk-definition}. There, using Sabot's polynomial decay Lemma~\ref{lem:sabot-field-moment-bound} \cite{Sabot2021}, we can bound \(\mathbb E[e^{\mathring u_y/2}]\). Fix \(w_0>0 \), Cauchy–Schwarz and the bound \(|f(i)-f(j)|\le2/r\) for the cutoff in \eqref{eq:annular-cutoff} then give
\[D_f(w)\le C_f(w_0)r^{-\alpha_{\mathrm{bulk}}(w_0)} \qquad(w\in[w_0,w_0+1]).\]
Substitution gives \(M_r'(w)\ge r^{\alpha_{\mathrm{bulk}}(w_0)}M_r(w)^2/((w_0+2)C_f(w_0))\). Dividing by \(M_r(w)^2\) and integrating over \([w_0,w_0+1]\) gives \eqref{eq:M-r-polynomial-bound}.

\section{The \(H^{2|2}\)-model on finite graphs}
We consider edge weights of the form \(W_{ij}=wa_{ij}\), where \(w>0\) is a common coupling parameter and \(W=(W_{ij})\) denotes the collection of edge weights. Throughout this section, \(\mathcal G=(V,E)\) is a finite connected simple graph, and \(a_{ij}=a_{ji}>0\) are fixed deterministic parameters for \(\{i,j\}\in E\). We set \(a_{ij}=0\) if \(\{i,j\}\notin E \), and \(a_{ii}=0\). We also regard \(W=(W_{ij})_{i,j\in V}\) as a symmetric matrix, with zero diagonal. All derivatives with respect to \(w\) are taken with \((a_{ij})\) fixed. The superscript \(W\) in \(\nu_V^W\) and \(\mu_{\delta,V}^W\) refers to the collection \(W=(wa_{ij})=(W_{ij})\).

The random-potential formulation on finite graphs was introduced in  \cite[Definition~1 and Theorem~3]{SabotTarresZeng2017}. For a vector \(\eta=(\eta_i)_{i\in V}\in[0,\infty)^V\), we use its extension given in \cite[Lemma~4]{Sabot2019}. This family is also studied as a multivariate reciprocal inverse Gaussian distribution in \cite[Sections~2 and~4]{LetacWesolowski2020}. Set
\[
H_\beta=2\operatorname{diag}\bigl((\beta_i)_{i\in V}\bigr)
-w(a_{ij})_{i,j\in V}.
\]
The law of the random potential \((\beta_i)_{i\in V}\) is the probability measure on \(\mathbb{R}_{+}^{V} \):
\begin{equation}\label{eq:joint-law-of-beta-finite-graph}
\nu_V^{W,\eta}(d\beta)
=\mathbf{1}_{H_\beta>0}
\left(\frac{2}{\pi}\right)^{|V|/2}
\frac{\exp\!\left(
-\frac12\langle\mathbf 1,H_\beta\mathbf 1\rangle
-\frac12\langle\eta,H_\beta^{-1}\eta\rangle
+\langle\eta,\mathbf 1\rangle\right)}
{\sqrt{\det H_\beta}}
\prod_{i\in V}d\beta_i,
\end{equation}
where \(\mathbf 1\) is the vector with all entries equal to one, \(\langle\cdot,\cdot\rangle\) is the Euclidean inner product on \(\mathbb R^V\), and \(H_\beta>0\) means that \(H_\beta\) is positive definite. In particular,
\[
-\frac12\langle\mathbf 1,H_\beta\mathbf 1\rangle
=-\sum_{i\in V}\beta_i+w\sum_{\{i,j\}\in E}a_{ij}.
\]
For \(\lambda=(\lambda_i)_{i\in V}\in[0,\infty)^V\), its Laplace transform is
\begin{equation}\label{eq:Laplace-transform-beta-with-eta}
\begin{aligned}
&\int e^{-\sum_{i\in V}\lambda_i\beta_i}\,\nu_V^{W,\eta}(d\beta)
=\prod_{i\in V}(1+\lambda_i)^{-1/2}\\
&\quad\times\exp\!\left(
-w\sum_{\{i,j\}\in E}a_{ij}
\bigl(\sqrt{(1+\lambda_i)(1+\lambda_j)}-1\bigr)
-\sum_{i\in V}\eta_i\bigl(\sqrt{1+\lambda_i}-1\bigr)
\right).
\end{aligned}
\end{equation}
We write \(\nu_V^W=\nu_V^{W,0}\). Unless otherwise specified, \(\eta\) is held fixed when differentiating with respect to \(w\).

Turning to the pinned effective field, take \(\eta=0\) and fix a vertex \(\delta\in V\). Under \(\nu_V^W=\nu_V^{W,0}\), let \(G=H_\beta^{-1}\) and define
\[
e^{u_i}=\frac{G(\delta,i)}{G(\delta,\delta)},\qquad i\in V,
\]
so that \(u_\delta=0\). By \cite[Section~3.1, Theorem~C]{Sabot2019}, the joint law of \((u_i)_{i\ne\delta}\) is
\begin{equation}\label{eq:joint-law-of-u-pinned-at-delta}
\mu_{\delta,V}^W(du)
=
\frac{1}{(2\pi)^{(|V|-1)/2}}
\exp\!\left(
-w\sum_{\{i,j\}\in E}a_{ij}\bigl(\cosh(u_i-u_j)-1\bigr)
\right)
\sqrt{D(W,u)}
\prod_{i\in V\setminus\{\delta\}}e^{-u_i}\,du_i,
\end{equation}
where
\begin{equation}\label{eq:spanning-tree-polynomial}
\begin{aligned}
D(W,u)
&=\sum_{T\in\mathcal T(\mathcal G)}
\prod_{\{i,j\}\in T}wa_{ij}e^{u_i+u_j}\\
&=w^{|V|-1}\sum_{T\in\mathcal T(\mathcal G)}
\prod_{\{i,j\}\in T}a_{ij}e^{u_i+u_j}.
\end{aligned}
\end{equation}
Here \(\mathcal T(\mathcal G)\) denotes the set of spanning trees of \(\mathcal G\). The probability measure \(\mu_{\delta,V}^W\) is the effective \(H^{2|2}\) measure on \(\mathcal G\), with edge weights \(wa_{ij}\) and pinning \(u_\delta=0\). Its density is the effective field density in \cite[Section~1.2, equations~(1.1)--(1.5)]{DisertoriSpencerZirnbauer2010}: take their field \(t=u\) on \(V\setminus\{\delta\}\), their edge couplings \(W_{ij}\), and their pinning parameters \(\varepsilon_i=W_{i\delta}\), and apply the matrix-tree theorem to their determinant. Its interpretation as the mixing measure of the time-changed VRJP is given in \cite[Theorem~2]{SabotTarres2015}, and its representation through Green-function ratios in \cite[Theorem~3]{SabotTarresZeng2017}.

\subsection{Marginal and conditional law of \(\beta \) field}
For \(\mu,\lambda>0\), let \(\operatorname{IG}(\mu,\lambda)\) denote the inverse-Gaussian distribution with mean \(\mu\) and shape \(\lambda\), whose density is
\[\operatorname{IG}(\mu,\lambda)(dt) =\sqrt{\frac{\lambda}{2\pi t^3}} \exp\!\left[-\frac{\lambda(t-\mu)^2}{2\mu^2t}\right]dt, \qquad t>0.\]
We write \(\operatorname{RIG}(\mu,\lambda)\) for the law of \(T^{-1}\) when \(T\sim\operatorname{IG}(\mu,\lambda)\). A change of variables gives the scaling identity
\begin{equation}\label{eq:inverse-gaussian-scaling}
T\sim\operatorname{IG}(\mu,\lambda)
\quad\Longrightarrow\quad
aT\sim\operatorname{IG}(a\mu,a\lambda),
\qquad a>0.
\end{equation}
In particular, for \(s>0\),
\begin{equation}\label{eq:inverse-gaussian-density}
\operatorname{IG}(1,s)(dt)
=\sqrt{\frac{s}{2\pi}}\,t^{-3/2}
\exp\!\left[-\frac{s}{2}(t+t^{-1}-2)\right]dt,
\qquad t>0.
\end{equation}
Its first two moments are
\begin{equation}\label{eq:inverse-gaussian-first-two-moments}
\int t\,\operatorname{IG}(1,s)(dt)=1,
\qquad
\int t^2\,\operatorname{IG}(1,s)(dt)=1+\frac1s.
\end{equation}
If \(T\sim\operatorname{IG}(1,s)\), then \(s/T\) has law \(\operatorname{RIG}(1/s,1)\), with density
\begin{equation}\label{eq:reciprocal-inverse-gaussian-density}
g_s(x)=\frac1{\sqrt{2\pi x}}
\exp\!\left[-\frac{(x-s)^2}{2x}\right],
\qquad x>0.
\end{equation}
We also use this formula for \(s=0\); in that case \(g_0\) is the Gamma density with shape \(1/2\) and rate \(1/2\).

We recall the marginal and conditional laws of the potential from  \cite[Lemma~5]{Sabot2019}; see also \cite[Propositions~4.2 and~4.3]{LetacWesolowski2020}. We extend the notation \(\nu_V^{W,\eta}\) to any symmetric matrix \(W\) with nonnegative entries, using the density \eqref{eq:joint-law-of-beta-finite-graph} with \(H_\beta=2\operatorname{diag}(\beta)-W\). Allowing nonzero diagonal entries in \(W\) amounts to a deterministic shift of the random potential: if \(W^\circ\) agrees with \(W\) off the diagonal and has zero diagonal, and \(\beta_i^\circ=\beta_i-W_{ii}/2\), then
\[
2\operatorname{diag}(\beta)-W
=2\operatorname{diag}(\beta^\circ)-W^\circ,
\qquad
\beta\sim\nu_V^{W,\eta}
\ \Longleftrightarrow\ 
\beta^\circ\sim\nu_V^{W^\circ,\eta}.
\]
For a matrix \(M\) indexed by \(V\) and subsets \(A,B\subset V\), write \(M_{AB}=(M_{ij})_{i\in A,j\in B}\) for its submatrix with rows in \(A\) and columns in \(B\). We abbreviate \(M_{A\{j\}}\) and \(M_{\{i\}B}\) to \(M_{Aj}\) and \(M_{iB}\), respectively, if \(A=B \), we write \(M_{A} \) for \(M_{AA} \) and let \(\mathbf 1_A\) denote the all-ones vector on \(A\). 

\begin{prop}\label{prop:potential-marginal-conditional-laws}
Let \(\beta\sim\nu_V^{W,\eta}\), where \(W\) is symmetric with nonnegative entries and \(\eta\in[0,\infty)^V\). Let \(\varnothing\ne U\subsetneq V\), and set \(U^c=V\setminus U\). Then
\begin{enumerate}
\item The marginal law of \(\beta_{U^c}=(\beta_i)_{i\in U^c}\) is
\begin{equation}\label{eq:beta-marginal-law}
\beta_{U^c}\sim\nu_{U^c}^{W_{U^cU^c},\widehat\eta},
\qquad
\widehat\eta=\eta_{U^c}+W_{U^cU}\mathbf 1_U.
\end{equation}
\item Given \(\beta_{U^c}\), define
\[
\widehat G_{U^c}=\bigl(2\operatorname{diag}(\beta_{U^c})-W_{U^cU^c}\bigr)^{-1},
\]
and set
\begin{equation}\label{eq:conditional-W-and-eta}
\check W=W_{UU}+W_{UU^c}\widehat G_{U^c}W_{U^cU},
\qquad
\check\eta=\eta_U+W_{UU^c}\widehat G_{U^c}\eta_{U^c}.
\end{equation}
Then, almost surely, the conditional law of \(\beta_U\) is
\begin{equation}\label{eq:beta-conditional-law}
\mathcal L(\beta_U\mid\beta_{U^c})
=\nu_U^{\check W,\check\eta}.
\end{equation}
In particular, \(\check{\eta}=0 \) when \(\eta=0 \). Setting
\[
\check H_U
=2\operatorname{diag}(\beta_U)-\check W
=(H_\beta)_{UU}-W_{UU^c}\bigl((H_\beta)_{U^cU^c}\bigr)^{-1}W_{U^cU},
\]
its conditional density is
\begin{equation}\label{eq:beta-conditional-density}
\left(\frac{2}{\pi}\right)^{|U|/2}
\frac{\exp\!\left(
-\frac12\langle\mathbf 1_U,\check H_U\mathbf 1_U\rangle
-\frac12\langle\check\eta,\check H_U^{-1}\check\eta\rangle
+\langle\check\eta,\mathbf 1_U\rangle
\right)}{\sqrt{\det\check H_U}}
\prod_{i\in U}d\beta_i
\end{equation}
on \(\{\check H_U>0\}\), and zero elsewhere.
\end{enumerate}
\end{prop}

\begin{coro}\label{cor:one-vertex-inverse-gaussian-laws}
Let \(\beta\sim\nu_V^{W,\eta}\), with \(W\) symmetric and nonnegative, \(\eta\in[0,\infty)^V\), and \(|V|\ge2\). Fix \(i\in V\), set \(B=V\setminus\{i\}\), and write \(G=H_\beta^{-1}\).

The random variable \(2\beta_i-W_{ii}\) has density \(g_{r_i}\) defined in \eqref{eq:reciprocal-inverse-gaussian-density}, where
\[r_i=\eta_i+\sum_{j\in B}W_{ij}.\]
If \(r_i>0\), then
\begin{equation}\label{eq:one-vertex-marginal-IG}
2\beta_i-W_{ii}\sim\operatorname{RIG}(1/r_i,1),
\qquad
\frac{r_i}{2\beta_i-W_{ii}}\sim\operatorname{IG}(1,r_i).
\end{equation}
In particular, when \(W_{ii}=0\) and \(r_i>0\), \(\beta_i\sim\operatorname{RIG}(2/r_i,2)\).

Given \(\beta_B\), set
\[\widehat G_B=\bigl((H_\beta)_{BB}\bigr)^{-1}, \qquad \check\eta_i=\eta_i+W_{iB}\widehat G_B\eta_B,\]
and
\[S_i=2\beta_i-W_{ii}-W_{iB}\widehat G_BW_{Bi}.\]
The conditional density of \(S_i\) is \(g_{\check\eta_i}\), and
\begin{equation}\label{eq:one-vertex-green-pinning-identity}
(G\eta)_i=\frac{\check\eta_i}{S_i}.
\end{equation}
Consequently, on \(\{\check\eta_i>0\}\),
\begin{equation}\label{eq:one-vertex-conditional-IG}
\mathcal L\bigl((G\eta)_i\mid\beta_B\bigr)
=\operatorname{IG}(1,\check\eta_i).
\end{equation}
On \(\{\check\eta_i=0\}\), one has \((G\eta)_i=0\).
\end{coro}

\begin{proof}
For a one-vertex law \(\nu_{\{i\}}^{w,r}\), the change of variables \(x=2\beta_i-w\) in \eqref{eq:joint-law-of-beta-finite-graph} gives the density
\[\frac1{\sqrt{2\pi x}} \exp\!\left(-\frac x2-\frac{r^2}{2x}+r\right) =g_r(x),\qquad x>0.\]
Taking \(U^c=\{i\}\) in Proposition~\ref{prop:potential-marginal-conditional-laws} therefore gives the marginal law. Taking \(U=\{i\}\) gives the conditional density of \(S_i\).

To identify \((G\eta)_i\), write \(v=G\eta\). The equations \(H_\beta v=\eta\) on \(B\) give
\[v_B=\widehat G_B(\eta_B+W_{Bi}v_i).\]
In particular, \(S_iv_i=\check\eta_i\), which gives \eqref{eq:one-vertex-green-pinning-identity}. The inverse-Gaussian laws follow from \eqref{eq:reciprocal-inverse-gaussian-density} and \eqref{eq:inverse-gaussian-scaling}.
\end{proof}

\subsection{Infinite graphs and wired boundary conditions}\label{subsec:infinite-graphs-wired-coupling}
On the graphs \(\widetilde V_n\) defined in Section~2, and denote their weight matrices by \(W^{(n)}\). Set
\[\eta_i^{(n)}=W^{(n)}_{i\delta_n} =\sum_{j\notin V_n}W_{ij}, \qquad i\in V_n.\]
By Proposition~\ref{prop:potential-marginal-conditional-laws}, the marginal law of \(\nu_{\widetilde{V}_{n}}^{W^{(n)}} \) on \(V_n\) is \(\nu_{V_n}^{W_{V_nV_n},\eta^{(n)}}\). These laws are consistent and give the infinite-volume law characterized by \eqref{eq:Laplace-transform-beta-on-U}.

The potentials can be realized on one probability space as follows. Sample \(\beta\sim\nu_V^W\) and, independently, \(\gamma\sim\operatorname{Gamma}(1/2,1)\). With \(\widehat{G}^{(n)}_\beta\) as in \eqref{eq:dirichlet-green-definition}, set
\begin{equation}\label{eq:wired-potential-coupling}
\beta_i^{(n)}=\beta_i\quad(i\in V_n),
\qquad
\beta_{\delta_n}^{(n)}
=\gamma+\frac12
\left\langle\eta^{(n)},
\widehat{G}^{(n)}_\beta\eta^{(n)}\right\rangle.
\end{equation}
Then \(\beta^{(n)}\) has law \(\nu_{\widetilde V_n}^{W^{(n)}}\), by the coupling construction in \cite[Section~4.3]{Sabot2019}. By Schur complement, the corresponding Green matrix satisfies
\begin{equation}\label{eq:wired-field-from-green-function}
G_\beta^{(n)}(\delta_n,\delta_n)=\frac1{2\gamma},
\qquad
e^{u_i^{(n)}}
=\frac{G_\beta^{(n)}(i,\delta_n)}
{G_\beta^{(n)}(\delta_n,\delta_n)}
=\left(\widehat{G}^{(n)}_\beta\eta^{(n)}\right)_i,
\quad i\in V_n.
\end{equation}
Thus \(u^{(n)}\), with \(u_{\delta_n}^{(n)}=0\), has law \(\mu_{\delta_n,\widetilde V_n}^{W^{(n)}}\). In particular, all these wired fields are functions of the same infinite-volume potential \(\beta\). By \cite[Theorem~1(ii)]{Sabot2019}, \(\exp(u_i^{(n)})\) converges almost surely to \(\psi(i)\). 

\subsection{Change of root}
Returning to finite graphs, we next describe the change of root for \(\mu_{\delta,V}^W\). For \(x\in V\), the field \(v_i=u_i-u_x\) satisfies \(v_x=0\). Comparing the densities in \eqref{eq:joint-law-of-u-pinned-at-delta} gives the following rerooting identity, see e.g. \cite{HuangZeng2026}.

\begin{lemm}\label{lem:rerooting}
For every \(x\in V\) and every nonnegative measurable function \(F\) on \(\{v\in\mathbb R^V:v_x=0\}\),
\begin{equation}\label{eq:rerooting}
\int e^{u_x}F\bigl((u_i-u_x)_{i\in V}\bigr)
\,\mu_{\delta,V}^W(du)
=
\int F(v)\,\mu_{x,V}^W(dv).
\end{equation}
In particular, \(\mathbb{E}_{\mu_{\delta,V}^{W}}(e^{u_{x}}) =1 \) for all \(x \).
\end{lemm}

\begin{proof}
The change of variables \(v_i=u_i-u_x\) has Jacobian \( 1\) and \(u_x=-v_\delta\). The cosh term in \eqref{eq:joint-law-of-u-pinned-at-delta} is unchanged, while
\[
D(W,u)=e^{-2(|V|-1)v_\delta}D(W,v),
\qquad
\sum_{i\in V}u_i=\sum_{i\in V}v_i-|V|v_\delta.
\]
Consequently, the image of \(\mu_{\delta,V}^W\) under this change of variables is \(e^{v_\delta}\mu_{x,V}^W(dv)\). Multiplication by \(e^{u_x}=e^{-v_\delta}\) gives \eqref{eq:rerooting}.
\end{proof}

\subsection{Distribution and moments of \(G(\delta,\delta) \)}
We next record distributional identities that give uniform fractional-moment bounds for the Green function in \cite[Section~3.1, Theorem~C]{Sabot2019}.

\begin{lemm}\label{lem:green-fractional-moments}
Let \(\beta\sim\nu_V^W\), and set \(G=H_\beta^{-1}\). Then
\begin{enumerate}
\item For every \(\delta\in V\), \(\gamma_\delta=\frac{1}{2G(\delta,\delta)} \) has the \(\text{Gamma}(\frac{1}{2},1) \) distribution, whose density is \(e^{-t}/\sqrt{\pi t}\) on \(t>0\). Moreover, \(G(\delta,\delta)\) is independent of \((\beta_i)_{i\ne\delta}\), and hence of the field
\(u_i=\log\frac{G(\delta,i)}{G(\delta,\delta)}. \)
\item For every deterministic vector \(q\in[0,\infty)^V\) with \(\sum_{i\in V}q_i=1\),
\(\frac{1}{2\langle q,Gq\rangle} \sim\operatorname{Gamma}\!\left(\frac12,1\right). \)

\item For \(0<s<1/2\), set \(
C_s=\frac{\Gamma(\frac12-s)}{2^s\sqrt{\pi}}.
\) Then, for every \(a\in V\),
\begin{equation}\label{eq:green-diagonal-moment}
\mathbb E_{\nu_V^W}\bigl[G(a,a)^s\bigr]=C_s.
\end{equation}
Furthermore, for all deterministic vectors \(\psi,\phi\in\mathbb R^V\),
\begin{equation}\label{eq:green-bilinear-moment}
\mathbb E_{\nu_V^W}
\bigl[|\langle\psi,G\phi\rangle|^s\bigr]
\le C_s\|\psi\|_1^s\|\phi\|_1^s,
\qquad
\|\psi\|_1=\sum_{i\in V}|\psi_i|.
\end{equation}
\end{enumerate}

Moreover, for any \(\eta\in[0,\infty)^V\), if \(\beta\sim\nu_V^{W,\eta}\), then, for nonnegative deterministic vectors \(p,q\),
\begin{equation}\label{eq:green-bilinear-moment-with-eta}
\mathbb E_{\nu_V^{W,\eta}}\!\left[
\bigl(p^TH_\beta^{-1}q\bigr)^s
\right]
\le C_s\|p\|_1^s\|q\|_1^s,
\qquad 0<s<\frac12.
\end{equation}
The bounds \eqref{eq:green-bilinear-moment} and \eqref{eq:green-bilinear-moment-with-eta} also hold when \(W\) is any symmetric matrix with nonnegative entries.
\end{lemm}

\begin{proof}
The first assertion is \cite[Section~3.1, Theorem~C]{Sabot2019}.

For the second assertion, use \eqref{eq:joint-law-of-beta-finite-graph} with \(\eta=tq\). Since \(\nu_V^{W,tq}\) is a probability measure,
\[
1=e^t\,
\mathbb E_{\nu_V^W}\left[
\exp\!\left(-\frac{t^2}{2}\langle q,Gq\rangle\right)\right],
\qquad t\ge0.
\]
Thus, writing \(Q=\langle q,Gq\rangle\), \(\mathbb E_{\nu_V^W}[e^{-\lambda Q}] =e^{-\sqrt{2\lambda}},\  \lambda\ge0. \)
This is the Laplace transform of \(1/(2\gamma)\), where \(\gamma\sim\operatorname{Gamma}(1/2,1)\). Consequently,
\[
\mathbb E[Q^s]
=\frac{2^{-s}}{\sqrt{\pi}}
\int_0^\infty t^{-s-1/2}e^{-t}\,dt
=C_s.
\]
Taking \(q\) to be the unit vector at \(a\) gives \eqref{eq:green-diagonal-moment}.

Finally, \(G\) has nonnegative entries and is positive definite, so
\[
|\langle\psi,G\phi\rangle|
\le \langle|\psi|,G|\phi|\rangle
\le
\langle|\psi|,G|\psi|\rangle^{1/2}
\langle|\phi|,G|\phi|\rangle^{1/2}.
\]
Cauchy--Schwarz give \eqref{eq:green-bilinear-moment}.

To obtain \eqref{eq:green-bilinear-moment-with-eta} for \(\eta\ne0\), append a new vertex \(\delta\) to \(V \) with edge weights \(W_{z\delta}=\eta_z\), and sample the potential law \(\nu_{V\cup\{\delta\}}^{W} \) on the enlarged graph. Its marginal on \(V\) is \(\nu_V^{W,\eta}\) by \eqref{eq:beta-marginal-law}.  \(H_\beta^{-1}\) is bounded entrywise by the \(V\times V\) block of the enlarged Green matrix. Applying \eqref{eq:green-bilinear-moment} to \(p,q\), extended by zero at \(\delta\), gives \eqref{eq:green-bilinear-moment-with-eta}.
\end{proof}

\section{Effective conductance estimates}
We consider the \(H^{2|2}\) model on a finite connected graph \(\mathcal G=(V,E)\) with positive edge weights \(W=(W_{ij})\). Fix a pinning vertex \(\delta\in V\) and another vertex \(a\ne\delta\). In the application below, \(\mathcal G\) will be the wired annular block \(\widetilde{\mathcal A}_{2^\ell r_0}\) of Definition~\ref{def:wired-annular-block}, with \(a=a_\ell\) and \(\delta=a_{\ell+1}\). Define
\[
Z(W)=\mathbb E_{\mu_{\delta,V}^W}[e^{u_a/2}].
\]

Let \(\widehat V=V\setminus\{a,\delta\}\), and assume \(|\widehat V|\ge2\). Throughout this section, write \(\mathbb E_{\beta_{\widehat V}}\) for expectation over
\begin{equation}\label{eq:interior-beta-law}
\beta_{\widehat V}\sim\nu_{\widehat V}^{W_{\widehat V\widehat V},\,W_{\widehat Va}+W_{\widehat V\delta}},
\end{equation}
and denote
\begin{equation}\label{eq:interior-green-matrix}
\widehat G=\left(2\operatorname{diag}(\beta_{\widehat V})-W_{\widehat V\widehat V}\right)^{-1}.
\end{equation}

\begin{lemm}\label{lem:Z-bound-by-M}
Define
\begin{equation}\label{eq:K-and-M-definition}
K=K(W,\beta_{\widehat V})=W_{a\delta}+W_{a\widehat V}\widehat G W_{\widehat V\delta},
\qquad M(W)=\mathbb E_{\beta_{\widehat V}}[K].
\end{equation}
We have \(Z(W)\le 1 \) and
\begin{equation}\label{eq:M-bound-by-weights-at-a}
0\le M(W)\le W_{a\delta}+\sum_{i\in\widehat V}W_{ai},
\end{equation}
moreover,
\begin{equation}\label{eq:Z-bound-by-M}
Z(W)\le
\left(\frac34\right)^{1/4}
\frac{\Gamma(1/4)}{2^{1/4}\sqrt\pi}\,M(W)^{1/4}.
\end{equation}
\end{lemm}

\begin{proof}
Identify \(a,\delta\) to a vertex \(\rho\), on the graph with vertex set \(\mathring V=\widehat V\cup\{\rho\}\), define edge weights
\begin{equation}\label{eq:weights-after-identifying-a-and-delta}
\mathring W_{ij}=W_{ij}\quad(i,j\in\widehat V),\qquad
\mathring W_{i\rho}=W_{ia}+W_{i\delta}\quad(i\in\widehat V),
\end{equation}
where the self-loop from identifying \(a,\delta\) is discarded. By Proposition~\ref{prop:potential-marginal-conditional-laws}, \eqref{eq:interior-beta-law} is the marginal law of both \(\nu_V^W\) and \(\nu_{\mathring V}^{\mathring W}\). The field \(\mathring u\) on \(\mathring V\) pinned at \(\rho\) satisfies
\begin{equation}\label{eq:contracted-field-from-green-function}
(e^{\mathring u_i})_{i\in\widehat V}
=\widehat G(W_{\widehat Va}+W_{\widehat V\delta}).
\end{equation}
Since \(\widehat G\) has nonnegative entries,
\[
0\le K\le W_{a\delta}+W_{a\widehat V}\widehat G(W_{\widehat Va}+W_{\widehat V\delta})
=W_{a\delta}+\sum_{i\in\widehat V}W_{ai}e^{\mathring u_i}.
\]
Taking expectations and using Lemma~\ref{lem:rerooting}, \(\mathbb E_{\beta_{\widehat{V}}}[e^{\mathring u_i}] = \mathbb{E}_{\nu_{\mathring V}^{\mathring W}}(e^{\mathring u_{i}})=1\) gives \eqref{eq:M-bound-by-weights-at-a}.

Under \(\nu_V^W\), condition on \(\beta_{\widehat V}\). Let \(\check W_{kl}=W_{kl}+W_{k\widehat V}\widehat G W_{\widehat Vl}\) for \(k,l\in\{a,\delta\}\). We have
\begin{equation}\label{eq:schur-complement-at-a-and-delta}
\check H_{\{a,\delta\}}
=\begin{pmatrix}
2\beta_a-\check W_{aa}&-K\\
-K&2\beta_\delta-\check W_{\delta\delta}
\end{pmatrix}.
\end{equation}
By Proposition~\ref{prop:potential-marginal-conditional-laws}, \(\mathcal L\!\left( \left(\beta_a-\tfrac12\check W_{aa}, \beta_\delta-\tfrac12\check W_{\delta\delta}\right) \,\middle|\,\beta_{\widehat V} \right) =\nu_{\{a,\delta\}}^{K}\), and by Schur complement,
\begin{equation}\label{eq:exp-u-a-from-schur-complement}
e^{u_a}=\frac{(H_\beta^{-1})_{a\delta}}{(H_\beta^{-1})_{\delta\delta}}
=\frac{K}{2\beta_a-\check W_{aa}}.
\end{equation}
By \eqref{eq:joint-law-of-u-pinned-at-delta},
\begin{equation}\label{eq:conditional-half-moment-bound-by-K}
\begin{aligned}
\mathbb E[e^{u_a/2}\mid\beta_{\widehat V}]
&=\sqrt{\frac K{2\pi}}\int_{\mathbb R}e^{-K(\cosh t-1)}\,dt\\
&\le\sqrt{\frac K{2\pi}}\int_{\mathbb R}e^{-Kt^4/24}\,dt
=\left(\frac34\right)^{1/4}
\frac{\Gamma(1/4)}{2^{1/4}\sqrt\pi}\,K^{1/4},
\end{aligned}
\end{equation}
Taking expectations and applying Jensen's inequality gives \eqref{eq:Z-bound-by-M}. By Cauchy--Schwarz and Lemma~\ref{lem:rerooting},
\begin{equation}\label{eq:Z-at-most-one}
Z(W)=\mathbb E_{\mu_{\delta,V}^W}[e^{u_a/2}]
\le \bigl(\mathbb E_{\mu_{\delta,V}^W}[e^{u_a}]\bigr)^{1/2}
=1.
\end{equation}
\end{proof}

To differentiate \(M(W)\) with respect to an edge weight \(W_{ij} \), we condition on the potentials \((\beta_{\ell})_{\ell\ne i,j} \) and take a Schur complement. The resulting conditional calculation involves the two-vertex law in the following lemma.
\begin{lemm}\label{lem:two-point-conductance-derivative}
Let \(c>0\), let \(p,q\in \mathbb{R}_{+}^2\), and \(\eta=p+q\), with \(\eta_1+\eta_2>0\). Set
\[
  \beta\sim\nu_{\{1,2\}}^{c,\eta},
\qquad
H=\begin{pmatrix}2\beta_1&-c\\-c&2\beta_2\end{pmatrix}.
\]
Define
\begin{equation}\label{eq:two-point-U-h-F-definition}
U=H^{-1}\eta,\qquad
h_k=\frac{(H^{-1}p)_k}{U_k}\quad(k=1,2),\qquad
F(c)=\mathbb{E}_{\nu_{\{1,2\}}^{c,\eta}}[p^{\mathsf T}H^{-1}q].
\end{equation}
Then \(U_1,U_2>0\), \(\mathbb{E}_{\nu_{\{1,2\}}^{c,\eta}}[U_1]=\mathbb{E}_{\nu_{\{1,2\}}^{c,\eta}}[U_2]=1\), and \(c\mapsto F(c)\) is finite and continuously differentiable on \((0,\infty)\), with \(p,q\) held fixed. Moreover,
\begin{equation}\label{eq:two-point-moment-bound-by-F-prime}
0\le F'(c)\le1,\qquad
\mathbb{E}_{\nu_{\{1,2\}}^{c,\eta}}\!\left[(U_1U_2)^2(h_1-h_2)^2\right]
\le\left(1+\frac1c\right)F'(c),
\end{equation}
and
\begin{equation}\label{eq:two-point-moment-comparison-7over4-2}
\mathbb{E}_{\nu_{\{1,2\}}^{c,\eta}}\!\left[(U_1U_2)^{7/4}(h_1-h_2)^2\right]
\le
\mathbb{E}_{\nu_{\{1,2\}}^{c,\eta}}\!\left[(U_1U_2)^2(h_1-h_2)^2\right].
\end{equation}
\end{lemm}

\begin{proof}
Since \(H>0\) and \(c>0\), every entry of \(H^{-1}\) is positive; hence \(U_1,U_2>0\). First assume \(\eta_1,\eta_2>0\). Write \(\Delta=p_1q_2-p_2q_1\). We have
\begin{equation}\label{eq:two-point-U-and-voltage-formulas}
U_1=\frac{2\beta_2\eta_1+c\eta_2}{\det H},\qquad
U_2=\frac{2\beta_1\eta_2+c\eta_1}{\det H},\qquad
U_1U_2(h_1-h_2)=\frac{\Delta}{\det H}.
\end{equation}

Applying \eqref{eq:one-vertex-marginal-IG} at vertex \(2\) gives
\[\frac{\eta_2+c}{2\beta_2}\sim\operatorname{IG}(1,\eta_2+c).\]
Conditionally on \(\beta_2\),
\[S_1=2\beta_1-\frac{c^2}{2\beta_2}, \qquad \check\eta_1=\eta_1+\frac{c\eta_2}{2\beta_2}.\]
By \eqref{eq:one-vertex-green-pinning-identity}, \(U_1=\check\eta_1/S_1\). Thus \eqref{eq:one-vertex-conditional-IG}, followed by exchanging the indices, gives
\begin{equation}\label{eq:two-point-inverse-gaussian-laws}
\begin{aligned}
&\mathcal L\!\left(\frac{\eta_2+c}{2\beta_2}\right)
=\operatorname{IG}(1,\eta_2+c),\\
&\mathcal L(U_1\mid\beta_2)
=\operatorname{IG}\!\left(1,\eta_1+\frac{c\eta_2}{2\beta_2}\right),\ \mathcal L(U_2\mid\beta_1)
=\operatorname{IG}\!\left(1,\eta_2+\frac{c\eta_1}{2\beta_1}\right).
\end{aligned}
\end{equation}
Using \eqref{eq:inverse-gaussian-first-two-moments} and taking expectations, we obtain
\[\mathbb E_{\nu_{\{1,2\}}^{c,\eta}}[U_1] =\mathbb E_{\nu_{\{1,2\}}^{c,\eta}}[U_2]=1.\]

By Corollary~\ref{cor:one-vertex-inverse-gaussian-laws}, with \(g_s\) defined in \eqref{eq:reciprocal-inverse-gaussian-density}, if
\[X=2\beta_2,\qquad Y=2\beta_1-\frac{c^2}{X}, \]
then \(X\) has density \(g_{\eta_2+c}\), and, \(Y\) has density \(g_{\eta_1+c\eta_2/X}\) given \(X\). Both densities are invariant under exchanging \(c\) and \(\eta_2\). Since \(\det H=XY\), its law also has this symmetry.

Thus, for \((\widetilde\beta_1,\widetilde\beta_2) \sim\nu_{\{1,2\}}^{\eta_2,(\eta_1,c)}\),
\[\mathbb{E}_{\nu_{\{1,2\}}^{c,\eta}}[(\det H)^k] = \mathbb{E}_{\nu_{\{1,2\}}^{\eta_2,(\eta_1,c)}} [(4\widetilde\beta_1\widetilde\beta_2-\eta_2^2)^k], \qquad k=-1,-2.\]
Define
\[
T=\frac{\eta_1+\eta_2}{2\widetilde\beta_1},\qquad
Z=\frac{2c\widetilde\beta_1+\eta_1\eta_2}
{4\widetilde\beta_1\widetilde\beta_2-\eta_2^2}.
\]

By \eqref{eq:two-point-inverse-gaussian-laws}, denote \(Q_c(t)=\eta_1\eta_2t+c(\eta_1+\eta_2)\), then \(4\widetilde\beta_1\widetilde\beta_2-\eta_2^2 =\frac{Q_{c}(T)}{TZ} \).
\[
T\sim\operatorname{IG}(1,\eta_1+\eta_2),\qquad
\mathcal L(Z\mid T)=\operatorname{IG}\!\left(1,\frac{Q_{c}(T)}{\eta_1+\eta_2}\right),
\]
we can integrate out \(Z \) using its conditional moments: \(\mathbb{E}(Z|T)=1 \) and \(\mathbb{E}(Z^2|T)= 1+ \frac{\eta_1+\eta_2}{Q_{c}(T)} \). \eqref{eq:inverse-gaussian-first-two-moments} yields
\begin{equation}\label{eq:two-point-inverse-determinant-moments}
\begin{aligned}
\mathbb{E}_{\nu_{\{1,2\}}^{c,\eta}}[(\det H)^{-1}]
&=\mathbb{E}\left( \frac{T}{Q_c(T)} \right)\\
\mathbb{E}_{\nu_{\{1,2\}}^{c,\eta}}[(\det H)^{-2}]
&=\mathbb{E}\left(\frac{T^2}{Q_c(T)^2}\left(1+\frac{\eta_1+\eta_2}{Q_c(T)}\right)\right).
\end{aligned}
\end{equation}

Since \(H^{-1}\) has nonnegative entries and \(q\le\eta\), we have \(0\le p^{\mathsf T}H^{-1}q\le p^{\mathsf T}U\), so \(F(c)\le p_1+p_2\). Using \(\eta=p+q\), we have
\begin{equation}\label{eq:two-point-conductance-decomposition}
p^{\mathsf T}H^{-1}q
=\frac{p_1q_1}{\eta_1}U_1
+\frac{p_2q_2}{\eta_2}U_2
+\frac{c\Delta^2}{\eta_1\eta_2\det H}.
\end{equation}
Taking expectations and inserting \eqref{eq:two-point-inverse-determinant-moments}, we obtain
\[
F(c)=\frac{p_1q_1}{\eta_1}+\frac{p_2q_2}{\eta_2}
+\frac{\Delta^2}{\eta_1\eta_2} \mathbb{E}\left( 
\frac{cT}{Q_c(T)} \right).
\]
Moreover,
\[
\frac{\partial}{\partial c}\frac{ct}{Q_c(t)}
=\frac{\eta_1\eta_2t^2}{Q_c(t)^2}.
\]
Differentiation under the integral sign gives
\begin{equation}\label{eq:F-prime-integral}
F'(c)=\Delta^2 \mathbb{E}\left( \frac{T^2}{Q_c(T)^2} \right).
\end{equation}
Since \(|\Delta|\le\eta_1\eta_2\) and \(Q_{c}(t) \ge \eta_{1}\eta_{2}t \), thus \(0\le F'(c)\le1\). Since \(1+(\eta_1+\eta_2)/Q_c(t)\le1+1/c\), combining this bound with \eqref{eq:two-point-U-and-voltage-formulas} and \eqref{eq:F-prime-integral} gives \eqref{eq:two-point-moment-bound-by-F-prime}.

By \eqref{eq:two-point-U-and-voltage-formulas}, \(U_1^2(h_1-h_2)^2 =\frac{\Delta^2}{(2\beta_1\eta_2+c\eta_1)^2}\) is \(\sigma(\beta_{1}) \)-measurable, and \eqref{eq:two-point-inverse-gaussian-laws} gives \(\mathbb E[U_2\mid\beta_1]=1 \le\mathbb E[U_2^2\mid\beta_1]\), conditioning on \(\beta_1\) yields
\[\mathbb{E}_{\nu_{\{1,2\}}^{c,\eta}}[U_1^2U_2(h_1-h_2)^2] \le \mathbb{E}_{\nu_{\{1,2\}}^{c,\eta}}[U_1^2U_2^2(h_1-h_2)^2].\]
We can exchange the indices in the above bound, using
\[
(U_1U_2)^{7/4}
\le\tfrac14U_1^2U_2+\tfrac14U_1U_2^2+\tfrac12U_1^2U_2^2.
\]
we deduce \eqref{eq:two-point-moment-comparison-7over4-2}.

Finally, if \(\eta_1=0<\eta_2\), the conditional laws \eqref{eq:two-point-inverse-gaussian-laws} remain valid and give \(\mathbb{E}_{\nu_{\{1,2\}}^{c,\eta}}[U_k]=1\) for \(k=1,2\). Since \(p_1=q_1=0\), we have \(h_1=h_2=p_2/\eta_2\) and \(F(c)=p_2q_2/\eta_2\). Thus \(F'(c)=0\), and both expectations in \eqref{eq:two-point-moment-comparison-7over4-2} vanish. The case \(\eta_2=0<\eta_1\) follows by symmetry.
\end{proof}

With \(\beta_{\widehat V}\) and \(\widehat G\) as in \eqref{eq:interior-beta-law}--\eqref{eq:interior-green-matrix}, let
\begin{equation}\label{eq:U-from-green-function}
U_i=\bigl[\widehat G(W_{\widehat Va}+W_{\widehat V\delta})\bigr]_i,
\quad i\in\widehat V,\qquad U_a=U_\delta=1,
\end{equation}
and define random conductances on \((V,E)\) by
\begin{equation}\label{eq:C-from-W-and-U}
C_{ij}=W_{ij}U_iU_j,\qquad \{i,j\}\in E.
\end{equation}
The variables \(U_i\) are positive, and \(U_i=e^{\mathring u_i}\) on \(\widehat V\) by \eqref{eq:contracted-field-from-green-function}.

\begin{lemm}\label{lem:M-differential-inequality}
Let \(W_e=wa_e\), where \(a_e\ge1\) are fixed and \(w>0\), and denote \(M(w)=M((wa_e)_{e\in E})\). For \(f:V\to\mathbb R\) with \(f_a=1\) and \(f_\delta=0\), set
\begin{equation}\label{eq:D-f-definition}
D_f(w)=\mathbb E_{\beta_{\widehat V}}\left[
\sum_{\{i,j\}\in E}W_{ij}(U_iU_j)^{1/4}(f_i-f_j)^2\right].
\end{equation}
The function \(M\) is continuously differentiable and nondecreasing on \((0,\infty)\), and
\begin{equation}\label{eq:M-differential-inequality}
M(w)^2\le(w+1)M'(w)D_f(w).
\end{equation}
\end{lemm}

\begin{proof}
Let \(h\) be the voltage for the conductances \eqref{eq:C-from-W-and-U}, with
\begin{equation}\label{eq:h-boundary-values-and-harmonicity}
h_a=1,\qquad h_\delta=0,\qquad
\sum_{j:\{i,j\}\in E}C_{ij}(h_i-h_j)=0\quad(i\in\widehat V).
\end{equation}
By \eqref{eq:U-from-green-function},
\[
2\beta_iU_i=\sum_{j\in\widehat V}W_{ij}U_j+W_{ia}+W_{i\delta}.
\]
Substituting \(C_{ij}=W_{ij}U_iU_j\) into \eqref{eq:h-boundary-values-and-harmonicity} gives
\[2\beta_iU_ih_i-\sum_{j\in\widehat V}W_{ij}U_jh_j=W_{ia}, \text{ i.e. }\widehat G^{-1}(U_ih_i)_{i\in\widehat V}=W_{\widehat Va}\]
together with \eqref{eq:U-from-green-function} gives
\begin{equation}\label{eq:Uh-and-U-one-minus-h}
(U_ih_i)_{i\in\widehat V}=\widehat G W_{\widehat Va},\qquad
(U_i(1-h_i))_{i\in\widehat V}=\widehat G W_{\widehat V\delta}.
\end{equation}
Consequently the current leaving \(a\) is
\begin{equation}\label{eq:K-equals-current-from-a}
\sum_{j:\{a,j\}\in E}C_{aj}(1-h_j)
=W_{a\delta}+W_{a\widehat V}\widehat G W_{\widehat V\delta}=K,
\end{equation}
so \(K\) in \eqref{eq:K-and-M-definition} is the effective conductance between \(a \) and \(\delta \).

For \(e\ne\{a,\delta\}\), let \(T=e\cap\widehat V\) and \(B=\widehat V\setminus T\). Proposition~\ref{prop:potential-marginal-conditional-laws} gives
\begin{equation}\label{eq:beta-B-marginal-law}
\mathcal L(\beta_B)=\nu_B^{W_{BB},\,W_{Ba}+W_{B\delta}+W_{BT}\mathbf1_T}.
\end{equation}
In particular, the law of \(\beta_{B} \) is independent of \(W_e\).

Fix an edge \(\{i,j\}\) with \(i,j\in \widehat{V} \), let \(B=\widehat V\setminus\{i,j\}\), and condition on \(\beta_B\). Define
\[\widehat G_B =\bigl(2\operatorname{diag}(\beta_B)-W_{BB}\bigr)^{-1}, \qquad \check W_{kl}=W_{kl}+W_{kB}\widehat G_BW_{Bl}, \quad k,l\in\{i,j,a,\delta\},\]
and
\[\check\beta_k =\beta_k-\tfrac12W_{kB}\widehat G_BW_{Bk}, \qquad k\in\{i,j\}.\]
Products involving \(B\) are understood to be zero when \(B\) is empty. Proposition~\ref{prop:potential-marginal-conditional-laws} gives
\[\mathcal L\bigl((\check\beta_i,\check\beta_j)\mid\beta_B\bigr) =\nu_{\{i,j\}}^{c,\eta},\]
where
\[c=\check W_{ij},\qquad p=\binom{\check W_{ia}}{\check W_{ja}},\qquad q=\binom{\check W_{i\delta}}{\check W_{j\delta}}, \qquad \eta=p+q.\]
Denote
\[\check H_{\{i,j\}} =\begin{pmatrix}2\check\beta_i&-c\\-c&2\check\beta_j\end{pmatrix},\]
we have
\[\binom{U_i}{U_j}=\check H_{\{i,j\}}^{-1}\eta, \qquad \binom{U_ih_i}{U_jh_j}=\check H_{\{i,j\}}^{-1}p, \qquad K=\check W_{a\delta}+p^{\mathsf T}\check H_{\{i,j\}}^{-1}q.\]
The quantities \(U_1,U_2,h_1,h_2,F(c)\) in Lemma~\ref{lem:two-point-conductance-derivative} are therefore identified with
\[U_i,\quad U_j,\quad h_i,\quad h_j,\quad \mathbb E[K\mid\beta_B]-\check W_{a\delta},\]
respectively. In particular, \(U_1U_2\) corresponds to \(U_iU_j=C_{ij}/W_{ij}\). Conditionally on \(\beta_B\), varying \(W_{ij}\) changes only \(c=\check W_{ij}\), with derivative \(1 \). Since the marginal law of \(\beta_B\) is independent of \(W_{ij}\), the bound \(0\le F'(c)\le1\) gives
\[\mathbb E_{\beta_B}[F'(c)] =\frac{\partial M(W)}{\partial W_{ij}}.\]
At least one coordinate of \(\eta\) is positive, since otherwise \(U_i=U_j=0\). Applying Lemma \ref{lem:two-point-conductance-derivative} conditionally on \(\beta_{B} \) and using \(c\ge W_{ij}\) gives
\begin{equation}\label{eq:conditional-edge-moment-bound}
\begin{aligned}
&0\le\frac{\partial}{\partial W_{ij}}\mathbb E[K\mid\beta_B]\le1,\\
&\mathbb E\left[\frac{C_{ij}^2}{W_{ij}^2}(h_i-h_j)^2\,\middle|\,\beta_B\right]
\le\left(1+\frac1{W_{ij}}\right)
\frac{\partial}{\partial W_{ij}}\mathbb E[K\mid\beta_B].
\end{aligned}
\end{equation}
Averaging \eqref{eq:conditional-edge-moment-bound} over the marginal law of \(\beta_B\) yields
\begin{equation}\label{eq:edge-moment-bound-by-M-derivative}
0\le
\mathbb E_{\beta_{\widehat V}}\left[
\frac{C_{ij}^2}{W_{ij}^2}(h_i-h_j)^2\right]
\le\left(1+\frac1{W_{ij}}\right)
\frac{\partial M(W)}{\partial W_{ij}},
\end{equation}
applying \eqref{eq:two-point-moment-comparison-7over4-2} gives
\begin{equation}\label{eq:edge-moment-comparison-7over4-2}
\mathbb E_{\beta_{\widehat V}}\left[
\left(\frac{C_{ij}}{W_{ij}}\right)^{7/4}(h_i-h_j)^2\right]
\le
\mathbb E_{\beta_{\widehat V}}\left[
\left(\frac{C_{ij}}{W_{ij}}\right)^2(h_i-h_j)^2\right].
\end{equation}

We now establish the same conclusions as in \eqref{eq:edge-moment-bound-by-M-derivative} and \eqref{eq:edge-moment-comparison-7over4-2} for an edge  \(e=\{a,i\}\). Let \(B=\widehat V\setminus\{i\}\), and define \(\widehat G_B\), \(\check W_{kl}\), and \(\check\beta_i\) by
\[\widehat G_B=\bigl(2\operatorname{diag}(\beta_B)-W_{BB}\bigr)^{-1}, \ \check W_{kl}=W_{kl}+W_{kB}\widehat G_BW_{Bl} , \ \check\beta_i=\beta_i-\tfrac12W_{iB}\widehat G_BW_{Bi}.\]
with \(k,l\in\{i,a,\delta\}\). Conditionally on \(\beta_B\), let \(p=\check W_{ia},\ q=\check W_{i\delta},\ s=p+q \). 
The Schur complement gives
\[U_i=\frac{s}{2\check\beta_i},\qquad U_ih_i=\frac{p}{2\check\beta_i},\qquad K=\check W_{a\delta}+\frac{pq}{2\check\beta_i} =\check W_{a\delta}+\frac{pq}{s}U_i.\]
By \eqref{eq:one-vertex-conditional-IG} and these identities,
\begin{equation}\label{eq:boundary-edge-conditional-identities}
\mathcal L(U_i\mid\beta_B)=\operatorname{IG}(1,s),\qquad
h_i=\frac ps,\qquad
\frac{\partial}{\partial W_{ai}}\mathbb E[K\mid\beta_B]=\frac{q^2}{s^2}.
\end{equation}
Indeed, \(\mathbb E[U_i\mid\beta_B]=1\) gives \(\mathbb E[K\mid\beta_B]=\check W_{a\delta}+pq/(p+q)\). With \(\beta_B\) fixed, \(p\) has derivative one with respect to \(W_{ai}\), whereas \(q\) and \(\check W_{a\delta}\) are independent of \(W_{ai}\); hence the derivative is \(q^2/(p+q)^2=q^2/s^2\).

Since \(s\ge W_{ai}\),
\[
\mathbb E[U_i^2(1-h_i)^2\mid\beta_B]
=\left(1+\frac1s\right)\frac{q^2}{s^2}
\le\left(1+\frac1{W_{ai}}\right)
\frac{\partial}{\partial W_{ai}}\mathbb E[K\mid\beta_B].
\]
This gives \eqref{eq:conditional-edge-moment-bound} for \(e=\{a,i\}\), since \(C_{ai}/W_{ai}=U_i\) and \(0\le q^2/s^2\le1\).

As the marginal law of \(\beta_B\) is independent of \(W_{ai}\) and \(0\le q^2/s^2\le1\), differentiation under the expectation gives
\[
\frac{\partial M(W)}{\partial W_{ai}}
=\mathbb E_{\beta_B}\!\left[\frac{q^2}{s^2}\right].
\]
Averaging the conditional estimate therefore gives \eqref{eq:edge-moment-bound-by-M-derivative} for \(e=\{a,i\}\).

The conditional version of \eqref{eq:edge-moment-comparison-7over4-2} for \(e=\{a,i\}\) follows from
\[
\mathbb E[U_i^{7/4}\mid\beta_B]
\le\bigl(\mathbb E[U_i^2\mid\beta_B]\bigr)^{7/8}
\le\mathbb E[U_i^2\mid\beta_B],
\]
where the last inequality uses \(\mathbb E[U_i^2\mid\beta_B]=1+1/s\ge1\). Multiplying by  \((1-h_i)^2\) and averaging over \(\beta_B\) gives \eqref{eq:edge-moment-comparison-7over4-2} for \(e=\{a,i\}\).

Exchanging \(a,\delta\) gives the case \(e=\{i,\delta\}\). For \(e=\{a,\delta\}\), \eqref{eq:K-and-M-definition} gives \(\partial M/\partial W_{a\delta}=1\), while \(C_{a\delta}/W_{a\delta}=h_a-h_\delta=1\), so both estimates hold.

Write \(\rho_W\) for the marginal density of \(\beta_B\), and set
\[D_e(W,\beta_B) =\frac{\partial}{\partial W_e}\mathbb E_W[K\mid\beta_B].\]
As \(W\to W'\) with positive limiting edge weights, dominated convergence in \eqref{eq:F-prime-integral}, and \eqref{eq:boundary-edge-conditional-identities} give
\[D_e(W,\beta_B)\longrightarrow D_e(W',\beta_B) \quad \rho_{W'}\text{-a.e.}, \qquad 0\le D_e\le1.\]
The marginal density in \eqref{eq:beta-B-marginal-law} satisfies
\[\rho_W\longrightarrow\rho_{W'}\quad\text{a.e.}, \qquad \|\rho_W-\rho_{W'}\|_{L^1}\longrightarrow0 \quad\text{by Scheff\'e's lemma}.\]
Consequently,
\[\begin{aligned} \left|\int D_e(W)\rho_W-\int D_e(W')\rho_{W'}\right| &\le \|\rho_W-\rho_{W'}\|_{L^1}\\ &\quad+\int |D_e(W)-D_e(W')|\rho_{W'} \longrightarrow0. \end{aligned}\]
Thus \(M\) is \(C^1\) and nondecreasing in each edge weight.

For each edge, choose an orientation \(e=(i,j)\) and write \(\nabla_e h=h_i-h_j\) and \(\nabla_e f=f_i-f_j\). Since \(W_e=wa_e\), \(a_e\ge1\), and all edge derivatives are nonnegative, \eqref{eq:edge-moment-bound-by-M-derivative} gives
\begin{equation}\label{eq:M-prime-bound-by-edge-sum}
\begin{aligned}
\mathbb E_{\beta_{\widehat V}}\sum_e\frac{C_e^2}{W_e}(\nabla_e h)^2
&\le\sum_e(wa_e+1)\frac{\partial M}{\partial W_e}\\
&\le(w+1)\sum_e a_e\frac{\partial M}{\partial W_e}
=(w+1)M'(w).
\end{aligned}
\end{equation}
By \eqref{eq:h-boundary-values-and-harmonicity} and \eqref{eq:K-equals-current-from-a}, summation by parts gives
\begin{equation}\label{eq:K-current-and-energy-identities}
K=\sum_e C_e\nabla_e h\nabla_e f.
\end{equation}
To obtain the fractional power in \(D_f(w)\), write
\[
C_e\nabla_e h\nabla_e f
=\left[\sqrt{W_e}\left(\frac{C_e}{W_e}\right)^{7/8}\nabla_e h\right]
\left[\sqrt{W_e}\left(\frac{C_e}{W_e}\right)^{1/8}\nabla_e f\right].
\]
Cauchy--Schwarz yields
\[
M(w)^2\le
\left(\mathbb E_{\beta_{\widehat V}}\sum_e
W_e\left(\frac{C_e}{W_e}\right)^{7/4}(\nabla_e h)^2\right)D_f(w).
\]
Finally, \eqref{eq:edge-moment-comparison-7over4-2} and \eqref{eq:M-prime-bound-by-edge-sum} bound the first factor by \((w+1)M'(w)\), giving \eqref{eq:M-differential-inequality}.
\end{proof}

\section{Polynomial bound on \(Z_{r}(w) \)}

We now bound \(Z_{r}(w_0) \) on the wired annular block \(\widetilde{\mathcal A}_r\) centered at \(x \) of Definition~\ref{def:wired-annular-block}, where \(r=2^\ell r_0\). Recall that this graph consists of the lattice points \(z \) with \(r+2\le |z-x|_1\le2r\), together with the vertices \(a=a_\ell\) and \(\delta=a_{\ell+1}\) attached to its inner and outer lozenges \(S_{r+2}(x) \) and \(S_{2r}(x) \), respectively; see Figures~\ref{fig:contracted-ball} and~\ref{fig:adjacent-annular-blocks}.

We choose a cutoff function \(f\) that equals \(1 \) near the inner lozenge \(S_{r+2}(x)\), vanishes near the outer lozenge \(S_{2r}(x)\), and varies linearly with the distance from the center \(x\) between these regions. Specifically, set \(f(a)=1\), \(f(\delta)=0\), and, for \(z\in V(\mathcal A_r)\), define
\begin{equation}\label{eq:annular-cutoff}
f(z)=
\begin{cases}
1,& |z-x|_1\le 5r/4,\\
\displaystyle\frac{7r/4-|z-x|_1}{r/2},
&5r/4<|z-x|_1<7r/4,\\
0,&|z-x|_1\ge 7r/4.
\end{cases}
\end{equation}
Since \(r\ge32\), this cutoff is constant across every edge incident to \(a \) or \(\delta \). On the remaining edges, \(|\nabla_e f|\le2/r\).

Write \(V=V(\widetilde{\mathcal A}_r)\) and \(E=E(\widetilde{\mathcal A}_r)\). For \(w>0\), give each edge \(e\in E\) weight \(W_e=wa_e\), where \(a_e\ge1\) is the number of original lattice edges represented by \(e\). Let \(\widehat V=V\setminus\{a,\delta\}=V(\mathcal A_r)\). Identify \(a\) and \(\delta\) to a single vertex \(\rho\), obtaining \(\mathring V=\widehat V\cup\{\rho\}\) with weights
\[
\mathring W_{ij}=W_{ij}\quad(i,j\in\widehat V),
\qquad
\mathring W_{i\rho}=W_{ia}+W_{i\delta}\quad(i\in\widehat V).
\]
Write \(\mathring E\) for the edge set of this graph, consisting of the unordered pairs \(\{i,j\}\) of distinct vertices in \(\mathring V\) with \(\mathring W_{ij}>0\).

Sample \(\beta_{\widehat V}\) with law \eqref{eq:interior-beta-law}, and define \(U,C\) by \eqref{eq:U-from-green-function}--\eqref{eq:C-from-W-and-U}. By \eqref{eq:contracted-field-from-green-function}, the field \(\mathring u_i=\log U_i\) on \(\widehat V\), with \(\mathring u_\rho=0\), has law \(\mu_{\rho,\mathring V}^{\mathring W}\). Recall that \(K\) in \eqref{eq:K-and-M-definition} is the effective conductance between \(a,\delta\) on the original annular block by \eqref{eq:K-equals-current-from-a}. Define
\begin{equation}\label{eq:M-r-definition}
M_r(w)=\mathbb E_{\beta_{\widehat V}}\!\left[K((wa_e)_e,\beta_{\widehat V})\right]
=M((wa_e)_{e\in E}).
\end{equation}
For an arbitrary orientation \(e=(i,j)\), put \(\nabla_e f=f(i)-f(j)\). \eqref{eq:D-f-definition} becomes
\begin{equation}\label{eq:D-f-on-annulus}
\begin{aligned}
D_f(w)
&=\mathbb E_{\beta_{\widehat V}}\!\left[
\sum_{e\in E}W_e\left(\frac{C_e}{W_e}\right)^{1/4}(\nabla_e f)^2\right]\\
&=w\sum_{\{i,j\}\in E(\mathcal A_r)}
\mathbb E_{\beta_{\widehat V}}\!\left[e^{(\mathring u_i+\mathring u_j)/4}\right]
(f(i)-f(j))^2.
\end{aligned}
\end{equation}
Here we use that \(\nabla_e f=0\) on edges incident to \(a\) or \(\delta\), and every edge of \(E(\mathcal A_r)\) has weight \(w\).

Lemmas~\ref{lem:Z-bound-by-M} and~\ref{lem:M-differential-inequality} give that, for fixed \(r \), \(w\mapsto M_r(w)\) is finite, continuously differentiable, and nondecreasing on \((0,\infty)\), and
\begin{equation}\label{eq:M-r-differential-inequality}
  M_r(w)^2\le(w+1)M_r'(w)D_f(w).
\end{equation}
Moreover, \(M_r(w)>0\), since the network is connected and its conductances are positive. For the field \(u\) on \(V\), pinned at \(u_\delta=0\), \eqref{eq:Z-bound-by-M} gives
\begin{equation}\label{eq:Z-r-bound-by-M-r}
  Z_r(w)=\mathbb E_{\mu_{\delta,V}^{(wa_e)_e}}[e^{u_a/2}]
  \le\left(\frac34\right)^{1/4}
  \frac{\Gamma(1/4)}{2^{1/4}\sqrt\pi}\,M_r(w)^{1/4}.
\end{equation}

\subsection{Mermin--Wagner estimate and bound on \(D_f(w)\)}
Fix \(w_0>0 \). We next bound \(D_f(w)\) uniformly for \(w\in[w_0,w_0+1]\). We use the following estimate from \cite[Lemma~2.1]{Sabot2021}, stated in our notation. See also \cite[Theorem~2]{KozmaPeled2021}.

\begin{lemm}\label{lem:sabot-field-moment-bound}
Let \(\mathcal G=(V,E)\) be a finite connected graph with positive edge weights \(W_{ij}\), and let \(\delta,y\in V\) be distinct vertices. Suppose that \(v:V\to\mathbb R\) satisfies \(v(\delta)=0\) and \(v(y)=1\). Let \(0<s<1\), \(q=(1-s)^{-1}\), and \(\gamma>0\). If
\[
q^2\gamma|v(i)-v(j)|\le\frac12
\qquad\text{for every }\{i,j\}\in E,
\]
then
\begin{equation}\label{eq:sabot-field-moment-bound}
\mathbb E_{\mu_{\delta,V}^W}\!\left[e^{su_y}\right]
\le
\exp\!\left[
-\gamma s+
\gamma^2q^2
\sum_{i\to j}(W_{ij}+1)(v(i)-v(j))^2
\right].
\end{equation}
The sum in \eqref{eq:sabot-field-moment-bound} includes both orientations of every edge.
\end{lemm}

Recall that \(r=2^\ell r_0\), where \(r_0\ge32\). To bound \(D_f(w)\) in \eqref{eq:D-f-on-annulus}, we bound the half-moments of \(e^{\mathring u_y}\) uniformly for \(w\in[w_0,w_0+1]\) and
\begin{equation}\label{eq:annular-bulk-definition}
 y\in\mathcal B_r
 :=\left\{z\in V(\mathcal A_r):
 \frac{5r}{4}-1\le |z-x|_1\le\frac{7r}{4}+1\right\}.
\end{equation}

\begin{prop}\label{prop:annular-half-moment-and-D-f-bounds}
Fix \(w_0>0\) and define
\begin{equation}\label{eq:alpha-bulk-and-C-f}
\alpha_{\mathrm{bulk}}(w_0)=\frac{1}{2048(w_0+2)},
\qquad
C_f(w_0)=72(w_0+1)16^{\alpha_{\mathrm{bulk}}(w_0)}.
\end{equation}
For every \(r=2^\ell r_0\), with \(r_0\ge32\) and \(\ell\ge0\), and every \(w\in[w_0,w_0+1]\), let \(\mathring u\) be the field on the graph \(\mathring V\), pinned at \(\rho\), with weights induced by \(W_e=wa_e\). With \(\mathbb E_{\beta_{\widehat V}}\) denoting expectation under \eqref{eq:interior-beta-law},
\begin{equation}\label{eq:annular-half-moment-bound-in-r}
\mathbb E_{\beta_{\widehat V}}\!\left[e^{\mathring u_y/2}\right]
\le16^{\alpha_{\mathrm{bulk}}(w_0)}r^{-\alpha_{\mathrm{bulk}}(w_0)},
\qquad y\in\mathcal B_r.
\end{equation}
For \(D_f(w)\) defined in \eqref{eq:D-f-on-annulus},
\begin{equation}\label{eq:D-f-polynomial-bound}
D_f(w)\le C_f(w_0)r^{-\alpha_{\mathrm{bulk}}(w_0)}.
\end{equation}
\end{prop}

\begin{proof}
Fix \(y\in\mathcal B_r\) and set \(N=\lfloor r/8\rfloor\ge4\). Define \(v_y:\mathring V\to\mathbb R\) by
\begin{equation}\label{eq:logarithmic-cutoff-v-y}
v_y(z)=
\frac{\bigl[\log(N/(1+|z-y|_1))\bigr]_+}{\log N}
\quad(z\in\widehat V),
\qquad v_y(\rho)=0,
\end{equation}
where \([b]_+=\max\{b,0\}\), so \(v_y(y)=1\). Since \(v_y\) vanishes at \(\rho\) and its neighbors, only edges in \(E(\mathcal A_r)\), all of weight \(w\), contribute to the sum in \eqref{eq:sabot-field-moment-bound}.
\[
|v_y(i)-v_y(j)|\le\frac{\log2}{\log N}
\qquad(\{i,j\}\in\mathring E).
\]
For each \(0\le n\le N-2\), there are \(16n+8\) oriented lattice edges joining the layers \(|z-y|_1=n\) and \(|z-y|_1=n+1\), counting both orientations. Using \(\log(1+b)\le b\), we obtain
\begin{equation}\label{eq:logarithmic-cutoff-energy-bound}
\begin{aligned}
\sum_{i\to j:\,\{i,j\}\in\mathring E}
(\mathring W_{ij}+1)(v_y(i)-v_y(j))^2
&\le\frac{w+1}{(\log N)^2}
\sum_{n=0}^{N-2}(16n+8)
\left(\log\frac{n+2}{n+1}\right)^2\\
&\le\frac{16(w+1)}{(\log N)^2}
\sum_{k=1}^{N-1}\frac1k
\le\frac{32(w+1)}{\log N}.
\end{aligned}
\end{equation}

We can apply Lemma~\ref{lem:sabot-field-moment-bound} on \((\mathring V,\mathring E)\), with pinning vertex \(\rho\), weights \(\mathring W\), and \(s=\frac12,\  q=2,\ \gamma=\frac{\log N}{512(w+1)}\), by checking the condition
\[
q^2\gamma|\nabla_e v_y|
\le\frac{\log2}{128(w+1)}<\frac12.
\]
We therefore obtain
\begin{equation}\label{eq:annular-half-moment-bound-in-N}
\begin{aligned}
\mathbb E_{\beta_{\widehat V}}\!\left[e^{\mathring u_y/2}\right]
&=\mathbb E_{\mu_{\rho,\mathring V}^{\mathring W}}\!\left[e^{\mathring u_y/2}\right]\le\exp\!\left[-\frac\gamma2
+4\gamma^2\frac{32(w+1)}{\log N}\right]\\
&=\exp\!\left[-\frac{\log N}{2048(w+1)}\right]
=N^{-1/(2048(w+1))}.
\end{aligned}
\end{equation}
Since \(w\le w_0+1\) and \(N\ge r/16\),
\[
N^{-1/(2048(w+1))}
\le N^{-\alpha_{\mathrm{bulk}}(w_0)}
\le16^{\alpha_{\mathrm{bulk}}(w_0)}r^{-\alpha_{\mathrm{bulk}}(w_0)},
\]
which gives \eqref{eq:annular-half-moment-bound-in-r}.

If \(\nabla_e f\ne0\) for \(e=\{i,j\}\), then \(i,j\in\mathcal B_r\). Cauchy--Schwarz and \eqref{eq:annular-half-moment-bound-in-r} give
\[
\begin{aligned}
\mathbb E_{\beta_{\widehat V}}\!\left[e^{(\mathring u_i+\mathring u_j)/4}\right]
&\le\left(
\mathbb E_{\beta_{\widehat V}}[e^{\mathring u_i/2}]
\mathbb E_{\beta_{\widehat V}}[e^{\mathring u_j/2}]
\right)^{1/2}\\
&\le16^{\alpha_{\mathrm{bulk}}(w_0)}r^{-\alpha_{\mathrm{bulk}}(w_0)}.
\end{aligned}
\]
At most \(18r^2\) edges verify \(0 < |\nabla_e f|\le2/r\). By \eqref{eq:D-f-on-annulus}, we obtain
\[
\begin{aligned}
D_f(w)
&\le w\,(18r^2)\frac4{r^2}
16^{\alpha_{\mathrm{bulk}}(w_0)}r^{-\alpha_{\mathrm{bulk}}(w_0)}\\
&\le72(w_0+1)16^{\alpha_{\mathrm{bulk}}(w_0)}r^{-\alpha_{\mathrm{bulk}}(w_0)}
=C_f(w_0)r^{-\alpha_{\mathrm{bulk}}(w_0)}.
\end{aligned}
\]
\end{proof}

\subsection{Decay of \(M_r(w)\) and \(Z_r(w)\)}

Fix \(w_0>0\). Combining \eqref{eq:M-r-differential-inequality} and \eqref{eq:D-f-polynomial-bound}, and using \(w+1\le w_0+2\), we obtain
\begin{equation}\label{eq:M-r-prime-lower-bound}
M_r'(w)\ge
\frac{r^{\alpha_{\mathrm{bulk}}(w_0)}}{(w_0+2)C_f(w_0)}
M_r(w)^2,
\qquad w\in[w_0,w_0+1].
\end{equation}
Since \(M_r\) is positive and continuously differentiable, division by \(M_r(w)^2\) and integration give
\[\frac{1}{M_r(w_0)}-\frac{1}{M_r(w_0+1)} = \int_{w_0}^{w_0+1}\frac{M_r'(w)}{M_r(w)^2}\,dw \ge \frac{r^{\alpha_{\mathrm{bulk}}(w_0)}}{(w_0+2)C_f(w_0)}.\]
As \(M_r(w_0+1)>0\), it follows that
\begin{equation}\label{eq:M-r-polynomial-bound}
M_r(w_0)\le
(w_0+2)C_f(w_0)r^{-\alpha_{\mathrm{bulk}}(w_0)}.
\end{equation}
Applying \eqref{eq:Z-r-bound-by-M-r} yields
\[Z_r(w_0)\le \frac{\Gamma(1/4)}{2^{1/4}\sqrt\pi} \left(\frac{3(w_0+2)C_f(w_0)}{4}\right)^{1/4} r^{-\alpha_{\mathrm{bulk}}(w_0)/4}.\]
Thus
\begin{equation}\label{eq:Z-r-polynomial-bound}
\mathbb E_{\mu_{\delta,V(\widetilde{\mathcal A}_r)}^{(w_0a_e)_e}}
\!\left[e^{u_{a_\ell}/2}\right]
=Z_r(w_0)
\le C_{w_0}(2^\ell r_0)^{-\alpha(w_0)},
\qquad r=2^\ell r_0,
\end{equation}
where
\begin{equation}\label{eq:Z-r-decay-constants}
\begin{aligned}
\alpha(w_0)
&=\frac{\alpha_{\mathrm{bulk}}(w_0)}{4}
=\frac{1}{8192(w_0+2)},\\
C_{w_0}
&=\frac{\Gamma(1/4)}{2^{1/4}\sqrt\pi}
\left(\frac{3(w_0+2)C_f(w_0)}{4}\right)^{1/4}.
\end{aligned}
\end{equation}
The bound \eqref{eq:Z-r-polynomial-bound} holds uniformly in \(\ell\ge0\) and \(r_0\ge32\).

\section{From annular bound to exponential decay}

\subsection{Factorization on \(\check{B}(x,R) \) and superpolynomial decay}

Fix \(r_0\ge32\) and let \(R=2^nr_0\), with \(n\ge1\). Start from \(B(x,R)\) with its exterior wired to \(\delta\), and contract each layer \(S_{2^\ell r_0+1}(x)\), for \(0\le\ell<n\), to \(a_\ell\). Set \(a_n=\delta\). The resulting graph \(\check B(x,R)\) is a chain of blocks: \(\mathcal B_0\) is the inner ball \(B(x,r_0)\) wired to \(a_0\), and recall Definition \ref{def:wired-annular-block}:
\[
\mathcal B_k=\widetilde{\mathcal A}_{2^{k-1}r_0},
\qquad 1\le k\le n.
\]
Write \(V_k=V(\mathcal B_k)\) for \(0\le k\le n\). Consecutive blocks intersect at a single vertex,
\[
V_{k-1}\cap V_k=\{a_{k-1}\},\qquad 1\le k\le n,
\]
nonconsecutive blocks are disjoint, and every edge belongs to exactly one block. All edge weights are inherited from the original lattice coupling \(W\), with parallel weights added and loops deleted; the superscript \(W\) on each field law below abbreviates this collection of weights.

For the field \(u\) on \(\check B(x,R)\), pinned at \(\delta=a_n\), define the relative field on each block by
\begin{equation}\label{eq:block-relative-fields}
u_z^{(k)}=u_z-u_{a_k},
\qquad z\in V_k,\quad 0\le k\le n.
\end{equation}

\begin{lemm}\label{lem:block-field-independence}
Under \(\mu_{\delta,V(\check B(x,R))}^W\), the fields \(u^{(0)},\ldots,u^{(n)}\) defined in \eqref{eq:block-relative-fields} are independent, and \(u^{(k)}\) has law \(\mu_{a_k,V_k}^W\) for every \(0\le k\le n\).
\end{lemm}

\begin{proof}
We use the density \eqref{eq:joint-law-of-u-pinned-at-delta}. The \(\cosh \) term splits over the blocks because the edge sets are disjoint and \(u_i-u_j=u_i^{(k)}-u_j^{(k)}\) on block \(k\).

Restricting a global spanning tree to the blocks, retaining their shared vertices, gives one spanning tree on each block; conversely, their union is a global spanning tree. Denote by \(D_{\mathcal B_k}\) the spanning-tree polynomial \eqref{eq:spanning-tree-polynomial} on block \(k\), and similarly by \(D_{\check B(x,R)}\) the global polynomial. Since \(u_i=u_i^{(k)}+u_{a_k}\) on block \(k\), each edge weight \(W_{ij}e^{u_i+u_j}\) contributes a factor \(e^{2u_{a_k}}\) relative to \(u^{(k)} \), and each block tree has \(|V_k|-1\) edges; hence
\begin{equation}\label{eq:spanning-tree-polynomial-by-blocks}
D_{\check B(x,R)}(W,u)
=
\prod_{k=0}^n
e^{2(|V_k|-1)u_{a_k}}
D_{\mathcal B_k}(W,u^{(k)}).
\end{equation}
Also
\[
\sum_{z\ne\delta}u_z
=
\sum_{k=0}^n
\left[
\sum_{z\in V_k\setminus\{a_k\}}u_z^{(k)}
+(|V_k|-1)u_{a_k}
\right],
\]
and the shifts in \(e^{-\sum_{z\ne\delta}u_z}\) cancel the factors in the square root of \eqref{eq:spanning-tree-polynomial-by-blocks}.

The change of variables \eqref{eq:block-relative-fields} has Jacobian \( 1\). Also, \(|V(\check B(x,R))|-1=\sum_{k=0}^n(|V_k|-1)\), so the powers of \(2\pi\) split. Thus,
\begin{equation}\label{eq:field-density-product-over-blocks}
d\mu_{\delta,V(\check B(x,R))}^W(u)
=
\prod_{k=0}^n d\mu_{a_k,V_k}^W(u^{(k)}).
\end{equation}
\end{proof}

Since \(u_{a_n}=0\), 
\[
u_x=u_x^{(0)}+\sum_{k=1}^n u_{a_{k-1}}^{(k)}.
\]
Lemma~\ref{lem:block-field-independence} therefore gives
\begin{equation}\label{eq:half-moment-product-over-blocks}
\mathbb E_{\mu_{\delta,V(\check B(x,R))}^W}
\!\left[e^{u_x/2}\right]
=
\mathbb E_{\mu_{a_0,V_0}^W}
\!\left[e^{u_x^{(0)}/2}\right]
\prod_{k=1}^n
\mathbb E_{\mu_{a_k,V_k}^W}
\!\left[e^{u_{a_{k-1}}^{(k)}/2}\right].
\end{equation}
The first factor is at most one by \eqref{eq:rerooting} and Cauchy--Schwarz. Applying \eqref{eq:Z-r-polynomial-bound} to each remaining factor yields
\begin{equation}\label{eq:half-moment-bound-by-dyadic-product}
\begin{aligned}
\mathbb E_{\mu_{\delta,V(\check B(x,R))}^W}
\!\left[e^{u_x/2}\right]
&\le\prod_{\ell=0}^{n-1}C_W(2^\ell r_0)^{-\alpha(W)}\\
&=C_W^n r_0^{-\alpha(W)n}
2^{-\alpha(W)n(n-1)/2}.
\end{aligned}
\end{equation}
The logarithm of the right-hand side equals
\[
-\frac{\alpha(W)\log2}{2}n^2
+n\left(\log C_W-\alpha(W)\log r_0
+\frac{\alpha(W)\log2}{2}\right).
\]
Since \(\log R=\log r_0+n\log2\), there are constants \(C,c>0\), depending only on \(W,r_0\), such that
\begin{equation}\label{eq:contracted-ball-half-moment-log-squared-decay}
\mathbb E_{\mu_{\delta,V(\check B(x,R))}^W}
\!\left[e^{u_x/2}\right]
\le C\exp\!\left[-c(\log R)^2\right],
\qquad R=2^nr_0.
\end{equation}

To pass from \(\check{B}(x,R) \) to the original wired ball, we use monotonicity theorem in \cite[Theorem~5 and 6]{Poudevigne2024}, also stated in \cite[Fact~2.2]{DisertoriMerklRolles2025}: increasing edge weights decreases the expectation of every convex function of \(e^{u_x}\). Applying this comparison to the convex function \(t\mapsto-\sqrt t\)  gives
\begin{equation}\label{eq:half-moment-comparison-under-contraction}
\mathbb E_{\mu_{\delta,V(B(x,R))\cup\{\delta\}}^W}
\!\left[e^{u_x/2}\right]
\le
\mathbb E_{\mu_{\delta,V(\check B(x,R))}^W}
\!\left[e^{u_x/2}\right].
\end{equation}
Together with \eqref{eq:contracted-ball-half-moment-log-squared-decay}, this gives the bound \(\mathbb E_{\mu_{\delta,V(B(x,R))\cup\{\delta\}}^W}[e^{u_x/2}]\le C e^{-c(\log R)^2}\).

\subsection{Lieb Simon argument}

To deduce exponential decay from the superpolynomial bound, we adapt the finite-size criterion of Simon and Lieb \cite{Simon1980,Lieb1980}, more precisely its fractional-moment counterpart in \cite[Section~4.a]{ASFH2001}, to the \(H^{2|2}\)-model.

Let \((\beta_z)_{z\in\mathbb Z^2}\) be the potential with nearest-neighbour coupling \(W>0\). Its restriction to a finite set \(D\subset\mathbb Z^2\) has law
\[
\beta_D\sim\nu_D^{W_{DD},\eta_D},
\qquad
\eta_D(z)=\sum_{v\notin D}W_{zv},
\]
where \(W_{zv}=W\mathbf1_{\{|z-v|_1=1\}}\). These finite-dimensional laws are consistent by \eqref{eq:beta-marginal-law}. All Green matrices below are formed from the same potential. For a finite region \(\Omega\), set
\[
H_\Omega=2\operatorname{diag}(\beta_\Omega)-W_{\Omega\Omega},
\qquad G_\Omega=H_\Omega^{-1}.
\]
For every finite \(D\subset\mathbb Z^2\), let \(\widehat G_D=(2\operatorname{diag}(\beta_D)-W_{DD})^{-1}\). 

Fix \(x,y\in\Omega\) with \(|x-y|_1\ge R+2\), and partition
\begin{equation}\label{eq:Omega-partition-into-A-S-C}
\Omega=A\sqcup S\sqcup C,
\qquad
A=\Omega\cap B(x,R),\quad
S=\Omega\cap S_{R+1}(x),\quad
C=\Omega\setminus B(x,R+1).
\end{equation}
Then \(x\in A\), \(y\in C\), and \(W_{AC}=0\). If \(S\) is empty, \(G_\Omega(x,y)=0\); hence we assume \(S\ne\varnothing\).

In the order \(A,S,C\), the matrix has the block form
\[
H_\Omega=
\begin{pmatrix}
H_A&-W_{AS}&0\\
-W_{SA}&H_S&-W_{SC}\\
0&-W_{CS}&H_C
\end{pmatrix}.
\]
Define
\begin{equation}\label{eq:schur-complement-on-S}
\begin{aligned}
\check W
&=W_{SS}
+W_{SA}\widehat G_AW_{AS}
+W_{SC}\widehat G_CW_{CS},\\
\check G_S
&=\bigl(2\operatorname{diag}(\beta_S)-\check W\bigr)^{-1}.
\end{aligned}
\end{equation}
Schur complement gives \(\check G_S=(G_\Omega)_{SS}\) and
\begin{equation}\label{eq:green-function-expansion-through-S}
\begin{aligned}
G_\Omega(x,y)
&=\bigl(\widehat G_AW_{AS}\check G_SW_{SC}\widehat G_C\bigr)_{xy}\\
&=\sum_{v,w\in S}
\left(\sum_{u\in A}\widehat G_A(x,u)W_{uv}\right)
\check G_S(v,w)
\left(\sum_{z\in C}W_{wz}\widehat G_C(z,y)\right).
\end{aligned}
\end{equation}
The first factor sums paths from \(x\) up to their first visit to \(S\), while the last sums paths after their last visit to \(S\). The middle factor includes all intervening excursions, including excursions back into \(A\) or \(C\).

We next integrate over the potential on \(S\). By Proposition~\ref{prop:potential-marginal-conditional-laws}, conditionally on \(\beta_A,\beta_C\), the law of \(\beta_S\) is \(\nu_S^{\check W,\check\eta}\), where
\begin{equation}\label{eq:conditional-eta-on-S}
\check\eta
=(\eta_\Omega)_S
+W_{SA}\widehat G_A(\eta_\Omega)_A
+W_{SC}\widehat G_C(\eta_\Omega)_C.
\end{equation}

The two vectors multiplying \(\check G_S\) in \eqref{eq:green-function-expansion-through-S} are nonnegative and fixed given \(\beta_A,\beta_C\). Applying \eqref{eq:green-bilinear-moment-with-eta}  therefore yields
\begin{equation}\label{eq:green-moment-conditioned-on-A-and-C}
\begin{aligned}
\mathbb E\!\left[
G_\Omega(x,y)^s\mid\beta_A,\beta_C
\right]
\le C_s
\left(\sum_{\substack{u\in A\\v\in S}}
\widehat G_A(x,u)W_{uv}\right)^s
\left(\sum_{\substack{w\in S\\z\in C}}
W_{wz}\widehat G_C(z,y)\right)^s.
\end{aligned}
\end{equation}
Because \(W_{AC}=0\), \(\beta_A\) and \(\beta_C\) are independent. Taking expectations in \eqref{eq:green-moment-conditioned-on-A-and-C} gives
\begin{equation}\label{eq:green-moment-bound-by-A-C-product}
\begin{aligned}
\mathbb E[G_\Omega(x,y)^s]
\le C_s\,
\mathbb E\!\left[
\left(\sum_{\substack{u\in A\\v\in S}}
\widehat G_A(x,u)W_{uv}\right)^s
\right]
\mathbb E\!\left[
\left(\sum_{\substack{w\in S\\z\in C}}
W_{wz}\widehat G_C(z,y)\right)^s
\right].
\end{aligned}
\end{equation}

For the first expectation in the r.h.s., define
\begin{equation}\label{eq:Psi-R-definition}
\Psi_R^x
=\bigl(\widehat G_{B(x,R)}\eta_{B(x,R)}\bigr)_x.
\end{equation}
Wiring the exterior of \(B(x,R)\) to \(\delta\) gives edge weights
\[
W_{z\delta}=\eta_{B(x,R)}(z)
=\sum_{v\notin B(x,R)}W_{zv}.
\]
For the field pinned at \(u_\delta=0\),
\[
H_{B(x,R)}(e^{u_z})_{z\in B(x,R)}=\eta_{B(x,R)},
\qquad
 e^{u_x}=\bigl(\widehat G_{B(x,R)}\eta_{B(x,R)}\bigr)_x=\Psi_R^x.
\]

If \(B(x,R+1)\subset\Omega\), then \(A=B(x,R)\), so
\[
\sum_{v\in S}W_{uv}=\eta_{B(x,R)}(u),\qquad u\in A.
\]
Consequently,
\[
\begin{aligned}
\sum_{\substack{u\in A\\v\in S}}
\widehat G_A(x,u)W_{uv}
&=\sum_{u\in B(x,R)}
\widehat G_{B(x,R)}(x,u)\eta_{B(x,R)}(u)=\Psi_R^x.
\end{aligned}
\]
For a general \(\Omega\), we have
\begin{equation}\label{eq:green-boundary-sum-bound-by-Psi-R}
\sum_{\substack{u\in A\\v\in S}}
\widehat G_A(x,u)W_{uv}
\le\Psi_R^x.
\end{equation}
Because \(\widehat G_A(x,u)\le\widehat G_{B(x,R)}(x,u)\).

For the second expectation, set
\[
\sigma_z=\sum_{w\in S}W_{wz},
\qquad z\in C.
\]
Subadditivity for \(0<s<1\) gives
\[
\mathbb E\!\left[
\left(\sum_{\substack{w\in S\\z\in C}}
W_{wz}\widehat G_C(z,y)\right)^s
\right]
\le
\sum_{z\in C}\sigma_z^s\,
\mathbb E[\widehat G_C(z,y)^s].
\]
Combining this with \eqref{eq:green-moment-bound-by-A-C-product} and \eqref{eq:green-boundary-sum-bound-by-Psi-R}, we obtain
\begin{equation}\label{eq:simon-lieb-fractional-moment-inequality}
\mathbb E[G_\Omega(x,y)^s]
\le
C_s\,\mathbb E[(\Psi_R^x)^s]
\sum_{z\in C}\sigma_z^s\,
\mathbb E[\widehat G_C(z,y)^s].
\end{equation}
Since \(\sigma_z>0\) requires \(z\in C\) to be a nearest neighbour of \(S\), it implies \(|z-x|_1=R+2\). The triangle inequality gives
\[
|z-y|_1\ge|x-y|_1-R-2,
\qquad
\sum_{z\in C}\sigma_z^s
\le W^s\bigl[4+4(R+1)2^s\bigr].
\]
For fixed \(W,s\), \(\sum_{z\in C}\sigma_z^s=O(R)\), while Jensen's inequality gives
\[
\mathbb E[(\Psi_R^x)^s]
\le
\left(\mathbb E[(\Psi_R^x)^{1/2}]\right)^{2s},
\qquad 0<s<\frac12.
\]
By \eqref{eq:half-moment-comparison-under-contraction} and \eqref{eq:contracted-ball-half-moment-log-squared-decay}, \(\mathbb E[(\Psi_R^0)^{1/2}]\le C e^{-c(\log R)^2}\). For each fixed \(0<s<1/2\) we can choose \(R\) so large that
\[
q_R:=C_sW^s\bigl[4+4(R+1)2^s\bigr]\,
\mathbb E[(\Psi_R^0)^s]<1.
\]
For integers \(d\ge0\), define
\[
F_s(d)=
\sup_{\substack{\Omega\subset\mathbb Z^2\text{ finite}\\
x,y\in\Omega,\ |x-y|_1\ge d}}
\mathbb E[G_\Omega(x,y)^s].
\]
The bound \eqref{eq:green-bilinear-moment-with-eta} gives \(F_s(d)\le C_s\). \eqref{eq:simon-lieb-fractional-moment-inequality} and translation invariance imply
\[
F_s(d)\le q_RF_s(d-R-2),
\qquad d\ge R+2.
\]
Iterating this inequality yields
\[
F_s(d)\le C_s q_R^{\lfloor d/(R+2)\rfloor}
\le \frac{C_s}{q_R}e^{-m_sd},
\qquad
m_s=\frac{\log(1/q_R)}{R+2}>0.
\]
Consequently,
\begin{equation}\label{eq:finite-dirichlet-green-exponential-decay}
\mathbb E[G_\Omega(x,y)^s]
\le \frac{C_s}{q_R}\exp\!\left(-m_s|x-y|_1\right),
\qquad x,y\in\Omega,
\end{equation}
uniformly over all finite regions \(\Omega\subset\mathbb Z^2\). 

\subsection{Infinite-volume limit}

Fix \(W>0\) and \(0<s<1/2\). The notation in \eqref{eq:dirichlet-green-definition} agrees with that of the finite-volume estimate:
\[
G_{V_n}=\widehat{G}^{(n)}_\beta.
\]
Thus \eqref{eq:finite-dirichlet-green-exponential-decay} gives the bound for \(\widehat{G}^{(n)}_\beta\). To bound the wired Green matrix \(G_\beta^{(n)}\), we estimate the second term in \eqref{eq:wired-dirichlet-green-decomposition}.

By \eqref{eq:wired-field-from-green-function},
\[
e^{u_x^{(n)}}
=\sum_{z\in\partial^-V_n}
\widehat{G}^{(n)}_\beta(x,z)\eta_z^{(n)},
\qquad 0\le\eta_z^{(n)}\le4W.
\]
Subadditivity and \eqref{eq:finite-dirichlet-green-exponential-decay} give
\begin{equation}\label{eq:wired-field-moment-decay-from-boundary}
\begin{aligned}
\mathbb E[e^{s u_x^{(n)}}]
&\le (4W)^s\frac{C_s}{q_R}
\sum_{z\in\partial^-V_n}e^{-m_s|x-z|_1}\\
&\le K_{W,s}\exp\!\left(-\frac{m_s}{2}d_{V_n}(x)\right),
\end{aligned}
\end{equation}
where \(K_{W,s}>0\) is independent of \(n\) and \(V_n\): indeed, \(|x-z|_1\ge d_{V_n}(x)-1\) for \(z\in\partial^-V_n\), and the tail of the lattice sum is bounded by a constant times \(e^{-m_s d_{V_n}(x)/2}\).

Set
\[
\theta_s=\frac{1-2s}{1-s}\in(0,1),
\qquad 2s=\theta_s s+(1-\theta_s).
\]
H\"older's inequality and \eqref{eq:wired-field-moment-decay-from-boundary} yield
\[
\mathbb E[e^{2s u_x^{(n)}}]
\le \mathbb E[e^{s u_x^{(n)}}]^{\theta_s}
\mathbb E[e^{u_x^{(n)}}]^{1-\theta_s}
\le K_{W,s}^{\theta_s}
 e^{-\theta_s m_s d_{V_n}(x)/2}.
\]
The field \(u^{(n)}\) is independent of \(\gamma\), and \(\mathbb E[(2\gamma)^{-s}]=C_s\). Hence Cauchy--Schwarz gives
\begin{equation}\label{eq:wired-green-correction-moment-bound}
\begin{aligned}
\mathbb E\!\left[\left(\frac{e^{u_x^{(n)}+u_y^{(n)}}}{2\gamma}\right)^s\right]
&\le C_s\left(
\mathbb E[e^{2s u_x^{(n)}}]\mathbb E[e^{2s u_y^{(n)}}]
\right)^{1/2}\\
&\le C_sK_{W,s}^{\theta_s}
\exp\!\left[-\frac{\theta_s m_s}{4}
\bigl(d_{V_n}(x)+d_{V_n}(y)\bigr)\right].
\end{aligned}
\end{equation}
Combining \eqref{eq:wired-dirichlet-green-decomposition}, \eqref{eq:finite-dirichlet-green-exponential-decay}, and \eqref{eq:wired-green-correction-moment-bound}, using \((a+b)^s\le a^s+b^s\) for \(a,b\ge0\), gives
\[
\mathbb E[G_\beta^{(n)}(x,y)^s]
\le C_{W,s}\left(
e^{-c_{W,s}|x-y|_1}
+e^{-c_{W,s}(d_{V_n}(x)+d_{V_n}(y))}
\right),
\]
with \(c_{W,s}=\theta_s m_s/4\) and
\(C_{W,s}=\max\{C_s/q_R,C_sK_{W,s}^{\theta_s}\}\).

Now let \(V_n\uparrow\mathbb Z^2\) be an exhaustion by finite connected sets. By \cite[Theorem~1(i)]{Sabot2019}, \(G_{V_n}(x,y)\) increases almost surely to \(\widehat G(x,y)\). Monotone convergence and \eqref{eq:finite-dirichlet-green-exponential-decay} give
\begin{equation}\label{eq:infinite-dirichlet-green-exponential-decay}
\mathbb E[\widehat G(x,y)^s]
=\lim_{n\to\infty}\mathbb E[G_{V_n}(x,y)^s]
\le\frac{C_s}{q_R}e^{-m_s|x-y|_1},
\qquad x,y\in\mathbb Z^2.
\end{equation}
The VRJP on \(\mathbb Z^2\) with constant edge weight \(W>0\) is recurrent by \cite{Sabot2021}. Hence \cite[Theorem~1(iv)]{Sabot2019} gives \(\psi\equiv0\) almost surely, so \eqref{eq:infinite-wired-green-definition} yields \(G=\widehat G\). Thus \eqref{eq:infinite-dirichlet-green-exponential-decay} also holds for \(G\), giving \eqref{eq:main-green-exponential-decay}.

{\bf Acknowledgments :}
J. Wang acknowledges valuable interactive discussions with GPT-5.5 and GPT-6 regarding the proof of the quantitative derivative bound (in particular, Eq. (66) of the first arxiv version). All steps were subsequently reinterpreted, refined, and verified by the authors, who take full responsibility for the mathematical correctness. X. Zeng thanks C. Sabot for many inspiring discussions and constant encouragement. 

We are also grateful to R. Poudevigne-Auboiron for his thoughtful and constructive comments on the first arxiv version, his suggestions help us to revise parts of the proofs in Section~4 and~5 and to clarify their underlying ideas, improving both the arguments and their exposition. We acknowledge the support of the Institut de recherche en mathématiques, interactions \& applications: IRMIA++.

\bibliographystyle{alpha}

\end{document}